\pdfoutput=1
\documentclass[letterpaper,11pt]{article}
\usepackage[margin=1in]{geometry}
\usepackage[CJKbookmarks=true,
            bookmarksnumbered=true,
            bookmarksopen=true,
            colorlinks=true,
            citecolor=red,
            linkcolor=blue,
            anchorcolor=red,
            urlcolor=blue
            ]{hyperref}
\usepackage[title]{appendix}
\usepackage{graphicx}
\usepackage{subfigure}
\usepackage{booktabs}
\usepackage{algorithm}
\usepackage{algorithmic}
\usepackage{amsmath,amsthm,amsfonts,amssymb}
\usepackage{hyperref}
\usepackage{color}
\usepackage{enumitem}
\usepackage{bm}
\usepackage{multirow}
\usepackage{bbm}
\usepackage{subfiles}
\usepackage{xspace}
\usepackage{caption}
\usepackage{mathtools} 
\usepackage{tikz}
\usepackage{natbib} 

\newtheorem{theorem}{Theorem}
\newtheorem{lemma}{Lemma}

\newtheorem{problem}{Problem}

\newtheorem{proposition}{Proposition}

\theoremstyle{definition}
\newtheorem{remark}{Remark}
\newtheorem{definition}{Definition}

\tikzset{global scale/.style={
    scale=#1,
    every node/.append style={scale=#1}
  }
}
\newcommand{\gaussianrho}{\mathcal{N}\pth{\begin{pmatrix} 0 \\ 0 \end{pmatrix}, \begin{pmatrix} 1 & \rho \\ \rho & 1 \end{pmatrix}}}

\newcommand{\maE}{\mathcal{E}}

\newcommand{\maH}{\mathcal{H}}
\newcommand{\maI}{\mathcal{I}}

\newcommand{\maN}{\mathcal{N}}

\newcommand{\maP}{\mathcal{P}}
\newcommand{\maQ}{\mathcal{Q}}

\newcommand{\maS}{\mathcal{S}}
\newcommand{\maT}{\mathcal{T}}

\newcommand{\prob}[1]{ \mathbb{P}\left[ #1 \right] }
\newcommand{\expect}[1]{ \mathbb{E}\left[ #1 \right] }

\newcommand{\ones}{\mathbbm{1}}

\newcommand{\en}{\mathsf{{e}}}

\newcommand{\pth}[1]{\left( #1 \right)}
\newcommand{\qth}[1]{\left[ #1 \right]}
\newcommand{\sth}[1]{\left\{ #1 \right\}}

\newcommand{\ti}{\tilde}
\newcommand{\TV}{\mathsf{TV}}
\newcommand{\Bin}{\mathrm{Bin}}

\newcommand{\sfC}{{\mathsf{C}}}

\newcommand{\sfP}{{\mathsf{P}}}

\newcommand{\HG}{\mathrm{HG}}

\newcommand{\ER}{Erd\H{o}s-R\'enyi }
\newcommand{\argmax}{\mathrm{argmax}}
\usetikzlibrary{intersections,calc,decorations.pathreplacing,through,arrows}
\usetikzlibrary{ decorations.markings,positioning}

\title{Sample Complexity of Correlation Detection in the Gaussian Wigner Model}
\author{Dong Huang and Pengkun Yang\thanks{
D.\ Huang and P.\ Yang are with the Department of Statistics and Data Science, Tsinghua
University. P. Yang is supported in part by the National
Key R\&D Program of China 2024YFA1015800.
}}

\begin{document}

\maketitle

\begin{abstract}
    
Correlation analysis is a fundamental step in uncovering meaningful insights from complex datasets. In this paper, we study the problem of detecting correlations between two random graphs following the Gaussian Wigner model with unlabeled vertices. Specifically, the task is formulated as a hypothesis testing problem: under the null hypothesis, the two graphs are independent, while under the alternative hypothesis, they are edge-correlated through a latent vertex permutation, yet maintain the same marginal distributions as under the null. We focus on the scenario where two induced subgraphs, each with a fixed number of vertices, are sampled. We determine the optimal rate for the sample size required for correlation detection, derived through an analysis of the conditional second moment. Additionally, we propose an efficient approximate algorithm that significantly reduces running time.

\end{abstract}


\section{Introduction}
Understanding the correlation between datasets is one of the most significant tasks in statistics.
In many applications, the observations may not be the familiar vectors but rather graphs.
Recently, there have been many studies on the problem of detecting graph correlation and recovering the alignments of two correlated graphs.
This problem arises across various domains:
\begin{itemize}
    \item In computer vision, 3-D shapes can be represented as graphs, where nodes are subregions and weighted edges represent adjacency relationships between different regions. A fundamental task for pattern recognition and image processing is determining whether two graphs represent the same object under different rotations \cite{berg2005shape,mateus2008articulated}.
    \item In natural language processing, each sentence can be represented as a graph, where nodes correspond to words or phrases, and the weighted edges represent syntactic and semantic relationships~\cite{hughes2007lexical}. The ontology alignment problem refers to uncovering the correlation between knowledge graphs that are in different languages~\cite{haghighi2005robust}.
    \item In computational biology, protein–protein interactions (PPI) and their networks are crucial for all biological processes. Proteins can be regarded as vertices, and the interactions between them can be formulated as weighted edges~\cite{singh2008global, vogelstein2015fast}.
\end{itemize}

Following the hypothesis testing framework proposed in \cite{barak2019nearly}, we formulate the graph correlation detection problem in Problem~\ref{prob:detection}. For a weighted graph $\mathbf{G}$ with vertex set $V(\mathbf{G})$ and edge set $E(\mathbf{G})$, the weight associated with each edge $uv$ is typically denoted as $\beta_{uv}(\mathbf{G})$ for any $u,v\in V(\mathbf{G})$. 

\begin{problem}\label{prob:detection}
    Let $\mathbf{G}_1$ and $\mathbf{G}_2$ be two weighted random graphs with vertex sets $V(\mathbf{G}_1),V(\mathbf{G}_2)$ and edge sets $E(\mathbf{G}_1),E(\mathbf{G}_2)$. Under the null hypothesis $\maH_0$, $\mathbf{G}_1$ and $\mathbf{G}_2$ are independent; under the alternative hypothesis $\maH_1$, there exists a correlation between  $E(\mathbf{G}_1)$ and $E(\mathbf{G}_2)$. Given $\mathbf{G}_1$ and $\mathbf{G}_2$, the goal is to test $\maH_0$ against $\maH_1$.
\end{problem}

A variety of studies have extensively investigated detection problems. However, the previous studies typically required full observation of all edges in $\mathbf{G}_1$ and $\mathbf{G}_2$ for detection, which is impractical when the entire graph is unknown in certain scenarios.
 In such cases, graph sampling---the process of sampling a subset of vertices and edges from the graph---becomes a powerful approach for exploring graph structure. This technique has been widely used in various settings, as it allows for inference about the graph without needing full observation \cite{leskovec2006sampling,hu2013survey}. In fact, there are several motivations leading us to consider the graph sampling method:
\begin{itemize}
    \item Lack of data. In social network analysis, the entire network is often unavailable due to API limitations. As a result, researchers typically select a subset of users from the network, which  essentially constitutes a sampling of vertices \cite{papagelis2011sampling}.
    \item Testing costs. The Protein Interaction Network is a common focus in biochemical research. However, accurately testing these interactions can be prohibitively expensive. As a result, testing methods based on sampled graphs are often employed \cite{stumpf2005subnets}.
    \item Visualization. The original graph is sometimes too large to be displayed on a screen, and sampling a subgraph provides a digestible representation, making it easier for visualization \cite{wu2016evaluation}.
\end{itemize}

In this paper, we consider sampling induced subgraphs for testing $\maH_0$ against $\maH_1$ when given two random graphs $\mathbf{G}_1$ and $\mathbf{G}_2$ with $|V (\mathbf{G}_1)| = |V(\mathbf{G}_2)|=n$. We randomly sample two induced subgraphs $G_1,G_2$ with $s$ vertices from $\mathbf{G}_1$ and $\mathbf{G}_2$, respectively. 
An induced subgraph of a graph is formed from a subset of the vertices of the graph, along with all the edges between them from the original graph.
Specifically, the sampling process for $G_1$ and $G_2$ is as follows: we first independently select vertex sets  $V(G_1)\subseteq V(\mathbf{G}_1)$ and $V(G_2)\subseteq V(\mathbf{G}_2)$ with $|V(G_1)| = |V(G_2)| =s$, and then
retain the weighted edge between $V(G_1)$ and $V(G_2)$ from the original graphs.
We assume $s\le n$ throughout the paper.

\subsection{Main Results}

In this subsection, we present the main results of the paper. 
Numerous graph models exist, with the Gaussian Wigner model being a prominent example \cite{ding2021efficient, fan2019spectral}, under which the weighted edges $\beta_{uv}(\mathbf{G})$ follow independent standard normals for any vertices $u,v\in V(\mathbf{G})$. This paper focuses on the Gaussian Wigner model with vertex set size $n$. Under the null hypothesis $\maH_0$, $\mathbf{G}_1$ and $\mathbf{G}_2$ follow independent Gaussian Wigner model with $n$ vertices. Under the alternative hypothesis $\maH_1$, $\mathbf{G}_1$ and $\mathbf{G}_2$ follow the following correlated Gaussian Wigner model.
\begin{definition}[Correlated Gaussian Wigner model]
    Let $\pi^*$ denote a latent bijective mapping from $V(\mathbf{G}_1)$ to $V(\mathbf{G}_2)$. We say that a pair of graphs $(\mathbf{G}_1,\mathbf{G}_2)$ are correlated Gaussian Wigner graphs if each pair of weighted edges $\beta_{uv}(\mathbf{G}_1)$ and $\beta_{\pi^*(u)\pi^*(v)}(\mathbf{G}_2)$ for any $u,v\in V(\mathbf{G}_1)$ are correlated standard normals with correlation coefficient $\rho\in (0,1)$.
\end{definition}

Let $\maQ$ and $\maP$ denote the probability measures for the sampled subgraphs $(G_1,G_2)$ under $\maH_0$ and $\maH_1$, respectively. We then focus on the following two detection criteria.

\begin{definition}[Strong and weak detection]\label{def:detection-criteria}
    We say a testing statistic $\maT =\maT(G_1,G_2)$ with a threshold $\tau$ achieves\begin{itemize}
        \item \emph{strong detection}, if the sum of Type I and Type II errors converges to 0 as $n\to\infty$:\begin{align*}
            \lim_{n\to\infty} \qth{\maP\pth{\maT<\tau}+\maQ\pth{\maT\ge \tau}} = 0;
        \end{align*}
        \item \emph{weak detection}, if the sum of Type I and Type II errors is bounded away from 1 as $n\to\infty$:\begin{align*}
            \lim_{n\to\infty} \qth{\maP\pth{\maT<\tau}+\maQ\pth{\maT\ge \tau}} <1.
        \end{align*}
    \end{itemize} 
\end{definition}

It is well-known that the minimal value of the sum of Type I and Type II errors between $\maP$ and $\maQ$ is $1-\TV(\maP,\maQ)$ (see, e.g., \cite[Theorem 7.7]{polyanskiy2025information}), achieved by the likelihood ratio test, where $\TV\pth{\maP,\maQ} = \frac{1}{2}\int|\,d\maP-\,d\maQ|$ is the total variation distance between $\maP$ and $\maQ$.
Thus strong and weak detection are equivalent to $\lim_{n\to \infty} \TV\pth{\maP,\maQ} = 1$ and $\lim_{n\to\infty} \TV\pth{\maP,\maQ} >0$, respectively.
We then establish the main results of correlation detection in the Gaussian Wigner model.

\begin{theorem}\label{thm:main-thm}
    There exist   constants $\overline{C},\underline{C}$ such that, for any $0<\rho<1$, if $s^2\ge \overline{C}\pth{\frac{ n\log n}{\log\pth{1/(1-\rho^2)}}\vee n}$, \begin{align*}
        \TV(\maP,\maQ) \ge 0.9.
    \end{align*}
    Moreover, if $s^2 = \omega(n)$, $\TV\pth{\maP,\maQ} = 1-o(1)$.
    
    Conversely, if $s^2\le  \underline{C}\pth{\frac{n\log n}{\log\pth{1/(1-\rho^2)}}\vee n}$,\begin{align*}
        \TV(\maP,\maQ)\le 0.1.
    \end{align*}
    Moreover, if $s^2 \le \frac{\underline{C}n\log n}{\log\pth{1/(1-\rho^2)}}$ or $s^2 = o(n)$, $\TV\pth{\maP,\maQ} = o(1)$.
\end{theorem}
The proof of Theorem~\ref{thm:main-thm} is deferred to Appendix~\ref{apd:main-thm}.
Theorem~\ref{thm:main-thm} implies that, for the hypothesis testing problem between $\maH_0$ and $\maH_1$ when sampling two induced subgraphs, the optimal rate for the sample size $s$ is of the order $\pth{\frac{n\log n}{\log\pth{1/(1-\rho^2)}}\vee n}^{1/2}$. Above this order, detection is possible, while below it, detection is impossible. Specifically, when $\frac{\overline{C}n\log n}{\log\pth{1/(1-\rho^2)}}> n^2$, the possibility condition requires $s>n$ in the above Theorem. However, we assume that the sample size $s\le n$, which indicates that there is no theoretical guarantee on detection, even when we sample the entire graph. Indeed, it is shown in~\cite{wu2023testing} that the detection threshold on $\rho$ in the fully correlation Gaussian Wigner model is $\rho^2 \asymp \frac{\log n}{n}$. Our results match the thresholds established in the previous work up to a constant for the special case $s=n$.

The possibility results can serve as a criterion for successful correlation detection in practice. For example, in computational biology, one may sample subgraphs to reduce testing costs, and the possibility results indicate when accurate detection remains feasible. 
Conversely, the impossibility results offer a theoretical tool for privacy protection. For instance, in social network de-anonymization, they imply that no test can succeed under certain conditions, thus providing a theoretical guarantee of privacy for anonymized networks.




\subsection{Related Work}

\paragraph{Graph matching}
The problem of graph matching refers to finding a  correspondence between the nodes of different graphs~\cite{caetano2007learning,livi2013graph}.
Recently, there have been many studies on the analysis of matching two correlated random graphs. In addition to the Gaussian Wigner model, another important model is the \ER model \cite{paul1959random}, where the edge follows Bernoulli distribution instead of normal distribution. 
As shown in~\cite{cullina2016improved, cullina2017exact,hall2023partial}, some sufficient and necessary conditions for the matching problem in the \ER model were provided.
The optimal rate for graph matching in the \ER model has been established in \cite{wu2022settling}, and the constant was sharpened by analyzing the densest subgraph in \cite{ding2023matching}. 
There are also many extensions on Gaussian Wigner model and correlated \ER graph model, including the inhomogeneous \ER model~\cite{racz2023matching, ding2023efficiently}, the partially correlated graphs model~\cite{huang2024information}, the correlated stochastic block model~\cite{chen2024computational,chen2025detecting}, the multiple correlated graphs model~\cite{ameen2024robust,ameen2025detecting}, and the correlated random geometric graphs model~\cite{wang2022random}.

\paragraph{Efficient algorithms and computational hardness}
There are many algorithms on the correlation detection and graph matching problem, including percolation graph matching algorithm \cite{yartseva2013performance}, subgraph matching algorithm \cite{barak2019nearly}, message-passing algorithm~\cite{piccioli2022aligning}, and spectral algorithm \cite{fan2019spectral},
while some algorithms may be computationally inefficient. There are also many efficient algorithms, based on the different correlation coefficient, including \cite{babai1980random, bollobas1982distinguishing, dai2019analysis, ganassali2020tree, ding2021efficient, mao2023exact, ding2023polynomial,mao2023random,araya2024seeded,ding2024polynomial,mao2024testing,ganassali2024statistical,muratori2024faster}.

The low-degree likelihood ratio~\cite{hopkins2017efficient,hopkins2018statistical} has emerged as a framework for studying computational hardness in high-dimensional statistical inference. It conjectures that polynomial-time algorithms succeed only in regimes where low-degree statistics succeed.
Based on the low-degree conjecture, the recent work by \cite{ding2023low, mao2024testing} established sufficient conditions for computational hardness results on the recovery and detection problems.

\subsection{Contributions and Outlines}
In this paper, we derive the optimal rate on sample size for correlation detection in the Gaussian Wigner model. Specifically, we prove that the optimal sample complexity is of rate $s\asymp \pth{\frac{n\log n}{\log(1/(1-\rho^2))}\vee n}^{1/2}$. We also propose a polynomial algorithm that significantly reduces computational cost.

In Sections~\ref{sec:upbd} and~\ref{sec:lwbd}, we prove the possibility results and impossibility results on sample size, respectively. Section~\ref{sec:alg} introduces our polynomial algorithm for correlation detection. Then, we run some numerical experiments in Section~\ref{sec:num-experiments} to verify the effectiveness for our algorithm proposed in Section~\ref{sec:alg}. Finally, Section~\ref{sec:future-direction} offers further discussion and future research directions, and the appendices contain detailed proofs and additional experimental results.




\section{Possibility Results}\label{sec:upbd}
In this section, we prove the possibility results in Theorem~\ref{thm:main-thm} by analyzing the error probability $\maP\pth{\maT<\tau}+\maQ\pth{\maT\ge \tau}$ under different regimes of $\rho$, which provides an upper bound for the optimal sample complexity.
Given a domain subset $S\subseteq V(G_1)$ and an injective mapping $\pi: S\mapsto V(G_2)$, along with a bivariate function $f:\mathbb{R}\times \mathbb{R}\mapsto \mathbb{R}$, we define the \emph{$f-$similarity graph $\maH_\pi^f$} as follows. The vertex set of $\maH_\pi^f$ is $V(\maH_\pi^f) = V(G_1)$, and for each edge $e$, the weighted edge is defined as  \begin{align}\label{eq:f-intersection}
    \beta_{e}(\maH_\pi^f) =
\begin{cases}
    f\pth{\beta_{e}(G_1),\beta_{\pi(e)}(G_2)}&\text{if }e\in\binom{S}{2}\\ 0&\text{otherwise}
\end{cases},
\end{align}
where $\pi(e)$ denotes the edge $\pi(u)\pi(v)$ for any edge $e=uv$.
Let $m\triangleq \frac{(1-\epsilon) s^2}{n}$ for some constant $0<\epsilon<1$, and denote $\maS_{s,m}$ as the set of injective mappings $\pi: S\subseteq V(G_1)\mapsto V(G_2)$ with $|S| = m$. Let $\en\pth{\maH_\pi^f}\triangleq \sum_{e\in E(G_1)} \beta_e\pth{\maH_\pi^f}$ define the sum of weighted edges in $\maH_\pi^f$. In fact, the quantity $\en(\maH_\pi^f)$ can be regarded as a similarity score between two graphs. Our test statistic takes  the form \begin{align}\label{eq:test-statistic}
    \maT(f) = \max_{\pi\in \maS_{s,m}} \en\pth{\maH_\pi^f} = \max_{\pi \in \maS_{s,m}} \sum_{e\in E(G_1)} \beta_e\pth{\maH_\pi^f}.
\end{align}
For simplicity, we write $\maT$ for $\maT(f)$ when the choice of $f$ is clear from the context. By the detection criteria in Definition~\ref{def:detection-criteria}, it suffices to bound the Type I error $\maQ\pth{\maT(f)\ge \tau}$ and the Type II error $\maP\pth{\maT(f)<\tau}$ for some appropriate threshold $\tau$. In the following, we outline a general recipe to derive an upper bound for error probabilities.

\paragraph{Type I error.} Under the null hypothesis $\maH_0$, the sampled subgraphs $G_1$ and $G_2$ are independent. Given a bivariate function $f$ and a threshold $\tau$, it should be noted that the distribution of the \emph{$f-$similarity graph} $\maH_\pi^f$ follows the same distribution for any $\pi \in \maS_{s,m}$. Consequently, applying the union bound yields that \begin{align*}
    \maQ\pth{\maT\ge \tau}\le |\maS_{s,m}| \maQ\pth{\en\pth{\maH_\pi^f} \ge \tau}.
\end{align*}
We then bound the tail probability by a standard Chernoff bound $$\maQ\pth{\en\pth{\maH_\pi^f} \ge \tau}\le \exp\pth{-\lambda \tau}\expect{\exp\pth{\lambda \en\pth{\maH_\pi^f}}}.$$ See~\eqref{eq:lambda-1} in Appendix~\ref{apd:proof-upbd-overlap} for more details.

\paragraph{Type II error.}
Under the alternative hypothesis $\maH_1$, recall that $\pi^*$ denotes the latent bijective mapping from $V(\mathbf{G}_1)$ to $V(\mathbf{G}_2)$. For the induced subgraphs $G_1, G_2$ sampled from $\mathbf{G}_1,\mathbf{G}_2$, we denote the set of common vertices as \begin{align}
    S_{\pi^*}&\triangleq V(G_1)\cap \pth{\pi^*}^{-1}(V(G_2)),\label{eq:def_of_S*}\\ T_{\pi^*}&\triangleq \pi^*(V(G_1))\cap V(G_2)\label{eq:def_of_T*}.
\end{align}
We note that the restriction of $\pi^*$ to $S_{\pi^*}$ is a bijective mapping between $S_{\pi^*}$ and $T_{\pi^*}$, and thus $|S_{\pi^*}| =|T_{\pi^*}|$. 
In our random sampling models, the vertices of $G_1$ and $G_2$ are independent and identically sampled without replacement from the two graphs $\mathbf{G}_1$ and $\mathbf{G}_2$, which yields the following Lemma regarding the sizes of $S_{\pi^*}$ and $T_{\pi^*}$.
\begin{lemma}\label{lem:hypergro-def}
    When randomly sampling vertex sets $V(G_1), V(G_2)$ from $V(\mathbf{G_1}),V(\mathbf{G}_2)$ with $|V(G_1)| = |V(G_2)| =s$, the size of common vertex set in~\eqref{eq:def_of_S*} follows a Hypergeometric distribution $\HG(n,s,s)$. Specifically, \begin{align*}
      \prob{|S_{\pi^*}| = t} =\binom{s}{t}\binom{n-s}{s-t}\bigg/\binom{n}{s},\text{ for any } t\in [s].
    \end{align*}
\end{lemma}
We then establish the main ingredients for controlling the Type II error. Under the distribution $\maP$, given $f$ and $\tau$,\begin{align}
    \nonumber \sth{\maT<\tau} =& \sth{\maT<\tau, |S_{\pi^*}|<m} \cup \sth{\maT<\tau, |S_{\pi^*}|\ge m}\\\label{eq:event_subseteq}
    \subseteq& \sth{|S_{\pi^*}|<m}\cup \sth{\maT<\tau,|S_{\pi^*}|\ge m}.
\end{align}
Since $\expect{|S_{\pi^*}|}=\frac{s^2}{n}>m$, the first event $\sth{|S_{\pi^*}|<m}$ can be bounded by the concentration inequality for Hypergeometric distribution in Lemma~\ref{lem:Hypergeometric_distribution_prop}. For the second event, it can be bounded by $\maP\pth{\maT<\tau\big|\,
  \vert S_{\pi^*}\vert\ge m}$. Under the event $\sth{|S_{\pi^*}|\ge m}$, there exists $\pi^*_m\in \maS_{s,m}$ such that $\pi^*_m = \pi^*$ on its domain set $\mathrm{dom}\pth{\pi^*_m}$. The error probability of the event $\sth{\maT<\tau, |S_{\pi^*}|\ge m}$ can be bounded by $\maP\pth{\en\pth{\maH_{\pi^*_m}^f}<\tau\big|\,\vert S_{\pi^*}\vert\ge m}$. We then use the concentration inequality to bound the tail probability. See~\eqref{eq:upbd-concentration-gau} for more details.

The quantity $\mathsf{e}(\mathcal{H}_\pi^f)$ measures the similarity score of a mapping $\pi$. Under the null hypothesis, $\mathsf{e}(\mathcal{H}_\pi^f)$ has a zero mean for all $\pi$, whereas under the alternative hypothesis, its mean with $\pi=\pi_m^*$ is strictly positive owing to the underlying correlation. We derive concentration inequalities to ensure that $\mathsf{e}\pth{\mathcal{H}_{\pi_m^*}^f}$ exceeds the maximum spurious score arising from stochastic fluctuations under the null, as shown in Propositions~\ref{prop: upbd-overlap} and~\ref{prop: upbd-least-square}.

\subsection{Detection by Maximal Overlap Estimator}
In this subsection, we use the test  statistic~\eqref{eq:test-statistic} with $f(x,y) = xy$ for possibility results. Indeed, this estimator is equivalent to maximizing the overlap on induced subgraphs between $G_1$ and $G_2$. Specifically, we have the following Proposition.
\begin{proposition}\label{prop: upbd-overlap}
    There exists a universal constant $C_1>0$ such that, for any $0<\rho<1$ and $\tau = \binom{m}{2}\frac{\rho}{2}$, if $s^2\ge \frac{C_1 n\log n}{\rho^2}$, \begin{align*}
        \maP\pth{\maT<\tau}+\maQ\pth{\maT\ge \tau} = o(1).
    \end{align*}
\end{proposition}

Proposition~\ref{prop: upbd-overlap} provides a sufficient condition on strong detection for any $0<\rho<1$. We refer readers to Appendix~\ref{apd:proof-upbd-overlap} for the detailed proof. Since $1-\TV(\maP,\maQ)\le \maP\pth{\maT<\tau}+\maQ\pth{\maT\ge \tau}$,
it achieves the optimal rate in Theorem~\ref{thm:main-thm} when $\rho = 1-\Omega(1)$. However, the rate is sub-optimal when $\rho=1-o(1)$. In fact, $s=2$ succeeds for detection when $\rho = 1$ by comparing the difference between all edges.
We will use a new estimator in Subsection~\ref{subsec:mini-least-square} to derive the optimal rate.

\subsection{Detection by Minimal Mean-Squared Error Estimator}\label{subsec:mini-least-square}

In this subsection, we use the test statistic~\eqref{eq:test-statistic} with $f(x,y) = -\frac{1}{2}(x-y)^2$ and  focus on the scenario where $\rho >1-e^{-6}$. Indeed, this estimator is equivalent to minimizing the mean squared error between the induced subgraphs of size  $m $ in  $G_1$  and  $G_2 $, respectively.
Indeed, the expected mean-square error for a correlated pair $\expect{\pth{\beta_e(G_1)-\beta_{\pi^*(e)}(G_2)}^2}$ is $2(1-\rho)$, while it stays bounded away from 1 for an uncorrelated pair. As a result, the choice of $f$ effectively distinguishes between $\maH_0$ and $\maH_1$ under strong signal condition.
We now state the following Proposition.

\begin{proposition}\label{prop: upbd-least-square}
    There exists a universal constant $C_2>0$ such that, for any $1-e^{-6}<\rho<1$ and $\tau = 2\binom{m}{2}(\rho-1)$, if $s^2\ge C_2\pth{\frac{n\log n}{\log\pth{1/(1-\rho)}}\vee n}$,
    \begin{align*}
        \maP\pth{\maT<\tau}+\maQ\pth{\maT\ge \tau}\le 0.1.
    \end{align*}
    Moreover, if $\frac{s^2}{n} = \omega(1)$, $\maP\pth{\maT<\tau}+\maQ\pth{\maT\ge \tau} = o(1)$.
\end{proposition}

We refer readers to Appendix~\ref{apd:upbd-least-square} for the detailed proof. Proposition~\ref{prop: upbd-least-square} provides sufficient conditions on strong and weak detection when $\rho$ is close to 1. This result fills the gap for the optimal rate of $s$ in Proposition~\ref{prop: upbd-overlap} when $\rho = 1-o(1)$. In view of  Propositions~\ref{prop: upbd-overlap} and~\ref{prop: upbd-least-square}, we note that $\rho^2 \asymp \log\pth{1/(1-\rho^2)}$ when $0<\rho\le 1-e^{-6}$ and $\log\pth{1/(1-\rho)}\asymp \log\pth{1/(1-\rho^2)}$ when $1-e^{-6}<\rho<1$. Then, there exists a universal constant $\overline{C}\ge C_1\vee C_2$ such that \begin{align*}
    \frac{\overline{C}}{\log\pth{1/(1-\rho^2)}}\ge \begin{cases}
        \frac{C_1}{\rho^2} &\text{if }0<\rho\le 1-e^{-6}\\
        \frac{C_2}{\log\pth{1/(1-\rho)}}&\text{if }1-e^{-6}<\rho<1
    \end{cases}.
\end{align*}
We note that $\frac{\overline{C} n\log n}{\rho^2}=\overline{C}\pth{\frac{n\log n}{\rho^2}\vee n} $ in Proposition~\ref{prop: upbd-overlap}, and thus proving the possibility results in Theorem~\ref{thm:main-thm}.  

\begin{remark}
    The possibility results can be extended to sub-Gaussian assumption on the weighted edges. The bound for the moment generating function holds under the sub-Gaussian assumption, and consequently, the Chernoff bound remains valid.
    See Remark~\ref{rmk:subgaussian} for more details.
\end{remark}

\begin{remark}
    In the previous work \cite{wu2023testing} on the correlated Gaussian Wigner model, the correlation exists over the entire graph. The maximal overlap estimator and the minimal mean-square error estimator over two graphs are equivalent since the sum of squares of the weighted edges is fixed. However, in our sampling model, the sum of squares of the weighted edges in the two subgraphs are random variables, and thus the two estimators differ. Indeed, the Maximum Likelihood Estimator (MLE) is $\max_{\pi\in \maS_{s,|S_{\pi^*}|}} \en\pth{\maH_\pi^f}$ with $f(x,y) = -\rho^2(x^2+y^2)+2\rho xy$, where $f(x,y)\asymp \rho xy$ when $\rho = 1-\Omega(1)$ and $f(x,y)\asymp -(x-y)^2$ when $\rho = 1-o(1)$. 
    The choice of different estimators reflects the use of MLE under different regimes. See~\eqref{eq:likelihood-func} in Appendix~\ref{apdsec:proof-of-partial-orbit} for details.
\end{remark}

\section{Impossibility Results}\label{sec:lwbd}
In this section, we establish the impossibility results for the detection problem, which provides a lower bound on the optimal sample complexity. We first present an overview of the proof. Recall that $S_{\pi^*}$ and $T_{\pi^*}$ are the sets of common vertices defined in~\eqref{eq:def_of_S*} and~\eqref{eq:def_of_T*}, respectively.
Under our sampling model, there exists a latent mapping between $S_{\pi^*}$ and $T_{\pi^*}$ under the hypothesis $\maH_1$.
When equipped with the additional knowledge of the common vertex sets, our problem can be reduced to detection with full observations on smaller correlated Gaussian Wigner model, the detection threshold for which is established in~\cite{wu2023testing}.
As shown in Lemma~\ref{lem:hypergro-def}, the size of $S_{\pi^*}$ and $T_{\pi^*}$ follows a hypergeometric distribution. Using the concentration inequality~\eqref{eq:concentration_for_hyper_2}, the size of $S_{\pi^*}$ satisfies $|S_{\pi^*}|\le (1+\epsilon) \expect{|S_{\pi^*}|}$ with high probability.
Therefore, the impossibility results from the previous work on full observations remain valid
when the number of correlated nodes is substituted with $(1+\epsilon)\expect{|S_{\pi^*}|}$.
However, such a reduction only proves tight when the correlation is weak. We will establish the remaining regimes by the conditional second moment method. 

For notational simplicity, we use $\TV(\maP,\maQ)$ to denote $\TV(\maP(G_1,G_2),\maQ(G_1,G_2))$ in this paper.
By \cite[Equation 2.27]{tsybakov2009introduction}, the total variation distance between $\maP $ and $\maQ $ can be upper bounded by the second moment: \begin{align}\label{eq:TV-tsy-upbd}
    \TV\pth{\maP,\maQ} \le \sqrt{\mathbb{E}_\maQ{\pth{\frac{\maP}{\maQ}}^2}-1}. 
\end{align}
The likelihood ratio is defined as 
\begin{align}
    \frac{\maP(G_1,G_2)}{\maQ(G_1,G_2)} = \frac{1}{n!} \sum_{\pi\in \maS_n} \frac{\maP(G_1,G_2|\pi)}{\maQ(G_1,G_2)}\label{eq:def_of_likeli_ratio},
\end{align}
 where $\maS_n$ denotes the set of  mappings $\pi:V(\mathbf{G}_1)\mapsto V(\mathbf{G}_2)$ between two original graphs.
Note that sometimes certain rare events under $\maP$ can cause the unconditional second moment to explode, while $\TV\pth{\maP,\maQ}$ remains bounded away from one. To circumvent such catastrophic events, we can compute the second moment conditional on such events. 

We consider the following event: 
\begin{align}\label{eq:condition1}
    \maE\triangleq \sth{(G_1,G_2,\pi):\vert \pi(V(G_1))\cap V(G_2)\vert \le \frac{(1+\epsilon) s^2}{n}}.
\end{align}
By Lemma~\ref{lem:hypergro-def}, the size of common vertex set $\vert \pi(V(G_1))\cap V(G_2)\vert$ follows hypergeometric distribution $\HG(n,s,s)$ under $\maP$.
In this paper, we define the conditional distribution as $\maP'(G_1,G_2,\pi) = \maP\pth{G_1,G_2,\pi|\maE}$. By Lemma~\ref{lem:Hypergeometric_distribution_prop}, we have $\maP\pth{\maE} = o(1)$ when $s=\omega\pth{n^{1/2}}$. Using $\TV(\maP,\maQ)\le \TV(\maP',\maQ)+o(1)$ and applying~\eqref{eq:TV-tsy-upbd} on $\maP'$ and $\maQ$ yields that a sufficient condition for $\TV(\maP,\maQ) = o(1)$ is $\mathbb{E}_\maQ\pth{\frac{\maP'}{\maQ}}^2 = 1+o(1)$. See~\eqref{eq:condi-second-intro} for more details.


\subsection{Weak correlation}
In this subsection, we show the impossibility results for weak correlation case where $0<\rho^2<n^{-1/2}$.
\begin{proposition}\label{prop:lwbd-sparse}
    For any $0<\rho^2<n^{-1/2}$, if $s^2 \le \frac{n \log n}{2\log\pth{1/(1-\rho^2)}}$, then $\TV(\maP,\maQ) =o(1)$.
\end{proposition}

We note that the total variation distance monotonically increases by the sample size $s$.
In view of Proposition~\ref{prop:lwbd-sparse}, we only need to tackle with the situation $s^2 = \frac{n \log n}{2\log\pth{1/(1-\rho^2)}}$, where $s=\omega\pth{n^{1/2}}$ since $\rho^2<n^{-1/2}$. 
 Therefore, a sufficient condition for $\TV(\maP,\maQ) = o(1)$ is $\TV\pth{\maP',\maQ} = o(1)$ by the triangle inequality. The proof of $\TV(\maP', \maQ) = o(1)$ can be reduced to the lower bound in \cite{wu2023testing} using a data processing inequality when given the common vertex sets $S_{\pi^*}$ and $T_{\pi^*}$. 
 Under weak correlation, the bottleneck is detecting the existence of latent mapping $\pi^*$. The detection is impossible even with the additional knowledge on the location of common vertices.
 The detailed proof of Proposition~\ref{prop:lwbd-sparse} is deferred to Appendix~\ref{apd:lwbd-aparse}.

\subsection{Strong correlation}
In this subsection, we present the impossibility results for strong correlation graphs where $n^{-1/2}\le \rho^2 <1$.
Let $\ti{\pi}$ be an independent copy of $\pi$. A key ingredient in the analysis of conditional second moment is the analysis of $\frac{\maP\pth{G_1,G_2|\pi}}{\maQ(G_1,G_2)} \frac{\maP\pth{G_1,G_2|\ti{\pi}}}{\maQ(G_1,G_2)}$.
We refer readers to Appendix~\ref{apd:lwbd-dense} for the details.



 We then analyze the terms $\frac{\maP(G_1,G_2|\pi)}{\maQ\pth{G_1,G_2}}$ and $\frac{\maP(G_1,G_2|\ti{\pi})}{\maQ\pth{G_1,G_2}}$. Recall the common vertex sets $S_{\pi}$ and $T_{\pi}$ defined in~\eqref{eq:def_of_S*} and~\eqref{eq:def_of_T*}, respectively. For any $e\notin \binom{S_\pi}{2}$ and $e'\notin \binom{T_\pi}{2}$, $\beta_e(G_1)$ and $\beta_{e'}(G_2)$ are independent under $\maP$, while under the null hypothesis distribution $\maQ$ they are also independent.
Therefore, the term $\frac{\maP(G_1,G_2|\pi)}{\maQ(G_1,G_2)}$ can be decomposed into  $\prod_{e\in \tbinom{S_\pi}{2}} \ell(\beta_e(G_1),\beta_{\pi(e)}(G_2))$,
where $ \ell(a,b)\triangleq \frac{\maP\pth{\beta_e(G_1) = a,\beta_{\pi(e)}(G_2) = b}}{\maQ\pth{\beta_e(G_1) = a,\beta_{\pi(e)}(G_2) = b}}$
for any $a,b\in \mathbb{R}$ is the ratio of density functions. 
We note that there are correlations between $\binom{S_\pi}{2},\binom{S_{\ti{\pi}}}{2},\binom{T_\pi}{2}$ and $\tbinom{T_{\ti{\pi}}}{2}$,  yielding that 
$\frac{\maP(G_1,G_2|\pi)}{\maQ(G_1,G_2)}$ and $\frac{\maP(G_1,G_2|\ti{\pi})}{\maQ(G_1,G_2)}$  are correlated. To deal with the correlation, our main idea is to decompose the edge sets into independent parts. 
To formally describe all correlation relationships, we use the \emph{correlated functional digraph} of two mappings $\pi$ and $\ti{\pi}$ between a pair of graphs introduced in~\cite{huang2024information}.

\begin{definition}[Correlated functional digraph]\label{def:correlated-functional}
    Let $\pi$ and $\ti{\pi}$ be two bijective mappings between $V(\mathbf{G}_1)$ and $V(\mathbf{G}_2)$ and $S_{\pi},T_{\pi},S_{\ti{\pi}},T_{\ti{\pi}}$ be the sets of common vertex defined in~\eqref{eq:def_of_S*} and~\eqref{eq:def_of_T*}.
    The correlated functional digraph of the functions $\pi$ and $\ti{\pi}$ is constructed as follows. Let the vertex sets be $\binom{S_\pi}{2}\cup \binom{S_{\ti{\pi}}}{2}\cup \binom{T_\pi}{2}\cup \binom{T_{\ti{\pi}}}{2}$. We first add every edge $e\mapsto \pi(e)$ for $e\in \binom{S_\pi}{2}$, and then merge each pair of nodes $(e,\ti{\pi}(e))$ for $e\in \binom{S_{\ti{\pi}}}{2}$ into one node.
\end{definition}

After merging all pairs of nodes, the degree of each vertex in the correlated functional digraph is at most two. Therefore, the connected components of the correlated functional digraph consist of paths and cycles. 
For example, for a path $(e_1,\pi(e_1),\cdots, e_j, \pi(e_j))$, where $e_1,\cdots,e_j$ are edges in $G_1$, we have $\ti{\pi}(e_2) = \pi(e_1),\cdots, \ti{\pi}(e_j) = \pi(e_{j-1})$; for a cycle $(e_1,\pi(e_1),\cdots, e_j, \pi(e_j))$, we have $\ti{\pi}(e_2) = \pi(e_1),\cdots, \ti{\pi}(e_j) = \pi(e_{j-1}), \ti{\pi}(e_1) = \pi(e_j)$. By decomposing the connected components, the analysis of edge sets is separated into independent parts.
Let $\sfP$ and $\sfC$ denote the collections of vertex sets belonging to different connected paths and cycles, respectively. For any $P\in \sfP$ and $C\in \sfC$, we define $\ell_e^\pi(G_1,G_2) = \ell\pth{\beta_e(G_1),\beta_{{\pi}(e)}(G_2)}$ and
\begin{align*}
    L_P \triangleq&\prod_{e\in \binom{S_\pi}{2}\cap P}\ell_e^\pi(G_1,G_2)\prod_{e\in \binom{S_{\ti{\pi}}}{2}\cap P}\ell_e^{\ti{\pi}}(G_1,G_2),\\L_C \triangleq&\prod_{e\in \binom{S_\pi}{2}\cap C}\ell_e^\pi(G_1,G_2)\prod_{e\in \binom{S_{\ti{\pi}}}{2}\cap C}\ell_e^{\ti{\pi}}(G_1,G_2).
\end{align*}


Note that the sets from $\sfP$ and $\sfC$ are disjoint. Consequently, for any $P,P'\in \sfP$ and $C,C'\in \sfC$, $L_P, L_{P'}, L_C$ and $L_{C'}$ are mutually independent. Furthermore, the expectations of $L_P$ and $L_C$ can be derived from the following Lemma.
\begin{lemma}\label{lem:partial_orbit_second_moment}
    For any $P\in \sfP,C\in \sfC$, we have $\mathbb{E}_\maQ\pth{L_P} = 1$ and $\mathbb{E}_\maQ\pth{L_C} =\frac{1}{1-\rho^{2|C|}}$. 
\end{lemma}
By Lemma~\ref{lem:partial_orbit_second_moment} and the joint independence between different paths and cycles, we have
\begin{align}
    \nonumber&\mathbb{E}_\maQ\qth{\frac{\maP(G_1,G_2|\pi)}{\maQ(G_1,G_2)} \frac{\maP(G_1,G_2|\ti{\pi})}{\maQ(G_1,G_2)}} \\=& \mathbb{E}_\maQ \qth{\prod_{P\in \sfP} L_P \prod_{C\in \sfC} L_C} = \prod_{C\in \sfC} \pth{\frac{1}{1-\rho^{2|C|}}}.\label{eq:second_moment_kappa}
\end{align}
 The cycles set $\sfC$ plays a key role in the analysis of conditional second moment. In order to analyze the properties of $\sfC$ in depth,
for any $\pi$ and $\ti{\pi}$, we define the \emph{core set} as
\begin{align}\label{eq:def_I^*}
    I^* \triangleq I^*(\pi,\ti{\pi}) \triangleq \cup_{C\in \sfC}\cup_{e\in C}\cup_{v\in V(e)\cap V(G_1)} v,
\end{align}
where $V(e)$ denotes the two vertices of edge $e$. Indeed, $I^*$ is the intersection set between $V(G_1)$ and all the vertices of edges in cycle set $\sfC$. In fact, the quantity $\prod_{C\in \sfC}\pth{\frac{1}{1-\rho^{2|C|}}}$ relies significantly on $I^*$.
We then show the following lemma on the properties of $I^*$.
\begin{lemma}[Properties of the \emph{core set}]\label{lem:property_of_I}
    For $I^*$ in \eqref{eq:def_I^*} and any $t\le s$, we have
    \begin{align*}
        I^* = \underset{I\subseteq V(G_1), \pi(I) = \ti{\pi}(I)}{\mathrm{argmax}} |I|,\quad \prob{|I^*| = t}\le \pth{\frac{s}{n}}^{2t}.
    \end{align*}
\end{lemma}

We then propose the following Proposition.

\begin{proposition}\label{prop:lwbd-dense}
    For any $n^{-1/2}\le \rho^2<1$, if $s^2 \le \frac{n\log n}{8\log\pth{1/(1-\rho^2)}}$, then $\TV\pth{\maP,\maQ} = o(1)$.
\end{proposition}

 The detailed proof of Proposition~\ref{prop:lwbd-dense} is deferred to Appendix~\ref{apd:lwbd-dense}.
 In the proof, we apply the conditional second moment method with the conditional distribution $\maP' = \maP(\cdot |\maE)$, where $\maE$ is defined in~\eqref{eq:condition1}. The analysis of the conditional second moment relies significantly on the decomposition of cycles and paths of a correlated functional digraph. By Lemma~\ref{lem:partial_orbit_second_moment}, the conditional second moment can be reduced to the calculation on cycles, while the vertex set induced by all cycles is exactly $I^*$. Combining this with the properties of $I^*$ in Lemma~\ref{lem:property_of_I}, we finish the proof of Proposition~\ref{prop:lwbd-dense}. In fact, under the strong correlation condition, detecting $\pi^*$ is no longer the bottleneck. We instead use a more delicate analysis based on the conditional second moment method.

By~\eqref{eq:concentration_for_hyper_2} in Lemma~\ref{lem:Hypergeometric_distribution_prop}, there exists $\underline{C}\le \frac{1}{8}$ such that, when $s^2\le \underline{C} n$, we have $\prob{|S_{\pi^*}| = 0}\ge 0.9$, which implies that $\TV(\maP,\maQ)\le 0.1$. Specifically, when $s^2 = o(n)$, $\prob{|S_{\pi^*}|=0} = 1-o(1)$, and thus $\TV\pth{\maP,\maQ} = o(1)$.
Combining this with Propositions~\ref{prop:lwbd-sparse} and~\ref{prop:lwbd-dense}, we prove the impossibility results in Theorem~\ref{thm:main-thm}.

\begin{remark}
    The second moment under our induced subgraph sampling model is equivalent to that on the vertex set induced by $I^*$. When fixing $I^*$, it is equal to the second moment of correlated Gaussian Wigner model with $\pi:I^*\to \pi(I^*)$. However, $I^* = I^*(\pi,\ti{\pi})$ is a random variable of $\pi,\ti{\pi}$, and hence a more thorough analysis on $I^*$ is needed, as shown in Lemma~\ref{lem:property_of_I}.
\end{remark}

\section{Algorithm}\label{sec:alg}
In this section, we present an efficient algorithm for detection. In Theorem~\ref{thm:main-thm}, we show that the estimator~\eqref{eq:test-statistic} achieves the optimal sample complexity for correlation detection. However, the estimator requires searching over $\maS_{s,m}$, with time complexity $\binom{s}{m}^2 \cdot m!$, resulting in poor performance for large graphs. Next, we propose an efficient algorithm to approximate the estimator in~\eqref{eq:test-statistic}. 

When the full observations of the graphs are given, there are many different efficient algorithms for detecting correlation and recovering graph matching. For instance, it is shown in~\cite{mao2023random, mao2024testing} that counting trees is an efficient way to detect correlation and recover graph matching when the correlation coefficient $\rho>\sqrt{\alpha}$, where $\alpha\approx 0.338$ is Otter's constant introduced in~\cite{otter1948number}. 
The message-passing algorithm~\cite{piccioli2022aligning,ganassali2024statistical} is also an efficient algorithm in the \ER model, which makes substantial use of the local tree structure.
Another approach for graph matching is   relaxing the original problem to a convex optimization problem~\cite{fan2019spectral}. 
Additionally, there are approaches based on initial seeds \cite{mossel2020seeded} and iterative methods \cite{ding2024polynomial} addressing this problem.

However, for the partial alignment problem and partial correlation detection problem, where only part of the original graphs are given, it becomes more challenging to find an efficient algorithm. One approach is to use deep learning techniques  \cite{jiang2022graph, wang2023deep, ratnayaka2024optimal}, while another way is to use low-degree structures, such as  cliques or trees \cite{sharma2018solving}. In this paper, we propose an algorithm that finds the initial seeds by matching the cliques, and then iteratively constructs the remaining mapping.
Three main components of our algorithm are outlined as follows. 

\begin{itemize}
    \item \emph{Match the small cliques.} Given two graphs $G_1$, $G_2$ and integers $K_1, N_1, N_2$ and a bivariate function $f$, we first randomly pick $N_1$ vertex set $V_1,\cdots, V_{N_1}\subseteq V(G_1)$ with $|V_1| = \cdots =|V_{N_1}| = K_1$ and $V_i\neq V_j$, for any $1\le i<j\le N_1$. For any $1\le i\le N_1$, define the injection $\pi_{i}:V_i\mapsto V(G_2)$ as \begin{align}\label{eq:def_of_pi_i}
    \pi_{i} \triangleq \underset{\substack{\pi:V_i\mapsto V(G_2)\\ \pi\text{ injection}}}{\mathrm{argmax}} \sum_{e\in \binom{V_i}{2}} \beta_e\pth{\maH_{\pi}^f},
\end{align}
where $\maH_{\pi_{i}}^f$ is the \emph{$f-$similarity graph} defined in~\eqref{eq:f-intersection}.
We then sort the values $\sum_{e\in \binom{V_i}{2}} \beta_e\pth{\maH_{\pi_i}^f}$ in decreasing order and select the top $N_2$ corresponding pairs of $(V_i, \pi_i)$. Without loss of generality, we assume that $(V_1,\pi_1),\cdots,(V_{N_2},\pi_{N_2})$ are the top $N_2$ pairs.
    \item \emph{Find seeds.} 
    Given an integer $K_2$,
    for any $U\subseteq [N_2]$ with $|U | = K_2$, let $V_U\triangleq \cup_{j\in U} V_j$.
    We say $U$ is \emph{compatible} if for any $v\in V_U$, $\pi_j(v)$'s are identical for all $j\in U$ such that $v\in V_{i_j}$. Let $\maI(U)$ denote the indicator function of \emph{compatible} set $U$. If $\maI(U)=1$, we define $\pi_U$ as the union of $\pi_j$ for any $j\in U$. Specifically, \begin{align*}
        \pi_U(v) = \pi_j(v) \text{ such that }v\in V_{i_j}, \text{ for any } v\in V_U.
    \end{align*} 
    The seed is then defined as \begin{align}\label{eq:seed-mapping}
        \pi_{0} = \underset{\pi_U:\maI(U)=1,U\subseteq [N_2],|U| = K_2}{\mathrm{argmax}} \sum_{e\in \binom{V_U}{2}} \frac{\beta_e\pth{\maH_{\pi_U}^f}}{\binom{|V_U|}{2}},
    \end{align}
    which maximize the average similarity score over $U$. 
    \item \emph{Iteratively construct mappings.} 
    Define the domain set and image set  of $\pi_0$ as $S_0$ and $T_0$, respectively. Then, we have $\pi_0: S_0\subseteq V(G_1)\mapsto T_0\subseteq V(G_2)$.
    Next, we iteratively extend the seed mapping by finding one vertex each from $V(G_1)$ and $V(G_2)$ until $|S_0| = |T_0| = m$. Specifically, given $\pi_0:S_0\mapsto T_0$, let \begin{align*}
        v_1,v_2 = \underset{\substack{v_1\in V(G_1)\backslash S_0\\ v_2\in V(G_2)\backslash T_0}}{\mathrm{argmax}} \sum_{v\in S_0} f\pth{\beta_{v_1 v}(G_1),\beta_{v_2 \pi_0(v)}(G_2)}.
    \end{align*}
    Then, we add the new mapping $v_1\mapsto v_2$ to $\pi_0$. This process is repeated iteratively, updating $\pi_0$ until $|S_0| = m$. Finally, we compute the test statistic $\sum_{e\in \binom{S_0}{2}} \beta_e\pth{\maH_{\pi_0}^f}$. $\maH_0$ is rejected if the test statistic exceeds the given threshold $\tau$, otherwise $\maH_0$ is accepted.
\end{itemize}

The detailed algorithm is shown in Algorithm~\ref{alg:clique-based}. 
Our algorithm comprises three main steps. In the first step, we select $N_1$ vertex sets $V_1,\cdots, V_{N_1}$ of size $K_1$ and search for injections $\pi_i$ from $V_i$ to $V(G_2)$, which requires $O(N_1\cdot s^{K_1})$ time. In the second step, we search over all subsets $U\subseteq [N_2]$ with $|U| = K_2$, which takes $O(N_2^{K_2})$ time. In the third step, we iteratively expand the mapping based on our seeds, which takes $O(m^2 s^2)$ time.
We typically choose $N_1\asymp s^{K_1}$ and $K_1\ge 3$, and thus the overall time complexity of the algorithm is $O(N_1\cdot s^{K_1}+N_2^{K_2})$.

Since only partial correspondence exists between the two graphs under $\maH_1$, finding the true mapping is challenging. 
We first use small cliques of size $K_1$ to trade accuracy for computational efficiency, although this often results in many incorrect mappings. To improve accuracy, we then test the compatibility of these small mappings and merge $K_2$ of them to construct a larger, more accurate mapping. This larger mapping is then used as a seed, and we iteratively enlarge it by adding one pair at a time until the size reaches $m$. This approach significantly reduces the running time compared to directly matching the larger cliques. 

As for the performance, a larger sample size $s$ leads to larger common vertex sets, and thus increases the number of correct mappings in Step 1. A larger $K_1$ corresponds to matching larger cliques in the first step. This increases the proportion of correct mappings within the $N_2$ candidate pairs when $K_1$ is below the size of common vertex sets. However, choosing $K_1$ beyond this size introduces wrong mappings. Besides, in the second step, we search over all $U\subseteq [N_2]$ with $|U| = K_2$ to identify the seeds. While a larger $K_2$ imposes a stricter matching criterion, choosing $K_2$ beyond the number of available correct mappings from Step 1 will degrade performance.


The accuracy and running time depend on $N_1, N_2, K_1$, and $ K_2$, and there is a trade-off between them: larger values of these parameters generally improve accuracy but increase the computational cost.
 

\begin{algorithm}[tb]
\caption{Clique-Based Detection Algorithm}
\label{alg:clique-based}
\begin{algorithmic}[1]
    \STATE{\textbf{Input:}} Two graphs $G_1,G_2$ with $s$ vertices, mapping size $m$, clique size $K_1$, combining size $K_2$, number of cliques $N_1$, number $N_2$, threshold $\tau$.
    \STATE{\textbf{Output:}} Detection result $\maH_0$ or $\maH_1$.
    \STATE Randomly select $N_1$ vertices sets $V_i\subseteq V(G_1)$ with $|V_i| = K_1$, for any $i=1,2,\cdots, N_1$.
    \STATE For each $V_i$, compute $\pi_i$ according to~\eqref{eq:def_of_pi_i}. Then, sort the values $\sum_{e\in \binom{V_i}{2}} \beta_e\pth{\maH_{\pi_i}^f}$ in descending order and select the top $N_2$ corresponding pairs of $(V_i,\pi_i)$. Without loss of generality, denote pairs as $(V_1,\pi_1),\cdots,(V_{N_2},\pi_{N_2})$.
    \STATE Find the seed mapping $\pi_0: S_0\subseteq V(G_1)\mapsto T_0\subseteq V(G_2)$ according to~\eqref{eq:seed-mapping}.
    \WHILE{$|S_0|<m$}
    \FOR{$v_1\in V(G_1)\backslash S_0$ and $v_2\in V(G_2)\backslash T_0$}
    \STATE Compute $\sum_{v\in S_0} f\pth{\beta_{v_1 v}(G_1),\beta_{v_2\pi_0(v)}(G_2)}$.
    \ENDFOR
    \STATE Find the pair $(v_1,v_2)$ for the maximal value of  $\sum_{v\in S_0} f\pth{\beta_{v_1 v}(G_1),\beta_{v_2\pi_0(v)}(G_2)}$ and add $v_1\mapsto v_2$ into $\pi_0$.
    \ENDWHILE
    \STATE Compute $\sum_{e\in \binom{S_0}{2}} \beta_e\pth{\maH_{\pi_0}^f}$, output $\maH_1$ if it exceeds $\tau$, otherwise output $\maH_0$.
\end{algorithmic}
\end{algorithm}

\section{Numerical Experiments}\label{sec:num-experiments}

In this section, we provide numerical results for Algorithm~\ref{alg:clique-based} on synthetic data. To this end, we independently generate 100 pairs of graphs that follow the independent Gaussian Wigner model, and another 100 pairs that follow the correlated Gaussian Wigner model with correlation $\rho$.

Fix $n=50,s=25, \rho =0.99,K_1 = 4, K_2 = 3, N_1 = 10000, N_2 = 500$, and $\epsilon = 0.01$. 
Then, $m = \left\lfloor\frac{(1-\epsilon)s^2}{n}\right\rfloor = 12$. In Figure~\ref{fig: histogram}, we plot the histogram of our approximated estimator $\sum_{e\in \binom{S_0}{2}} \beta_e\pth{\maH_{\pi_0}^f}$ defined in Algorithm~\ref{alg:clique-based}. We see that the histograms under the independent model and the correlated model are well-separated. By picking an appropriate threshold $\tau$, the proposed algorithm succeeds in correlation detection. We note that when $K_1 = 2$ and $K_2=1$, Algorithm~\ref{alg:clique-based} is equivalent to comparing the pairwise differences of all edges, while our approach with $K_1=4$ and $K_2=3$ is more effective than this trivial method.

\begin{figure}
    \centering
    \includegraphics[width=0.5\linewidth]{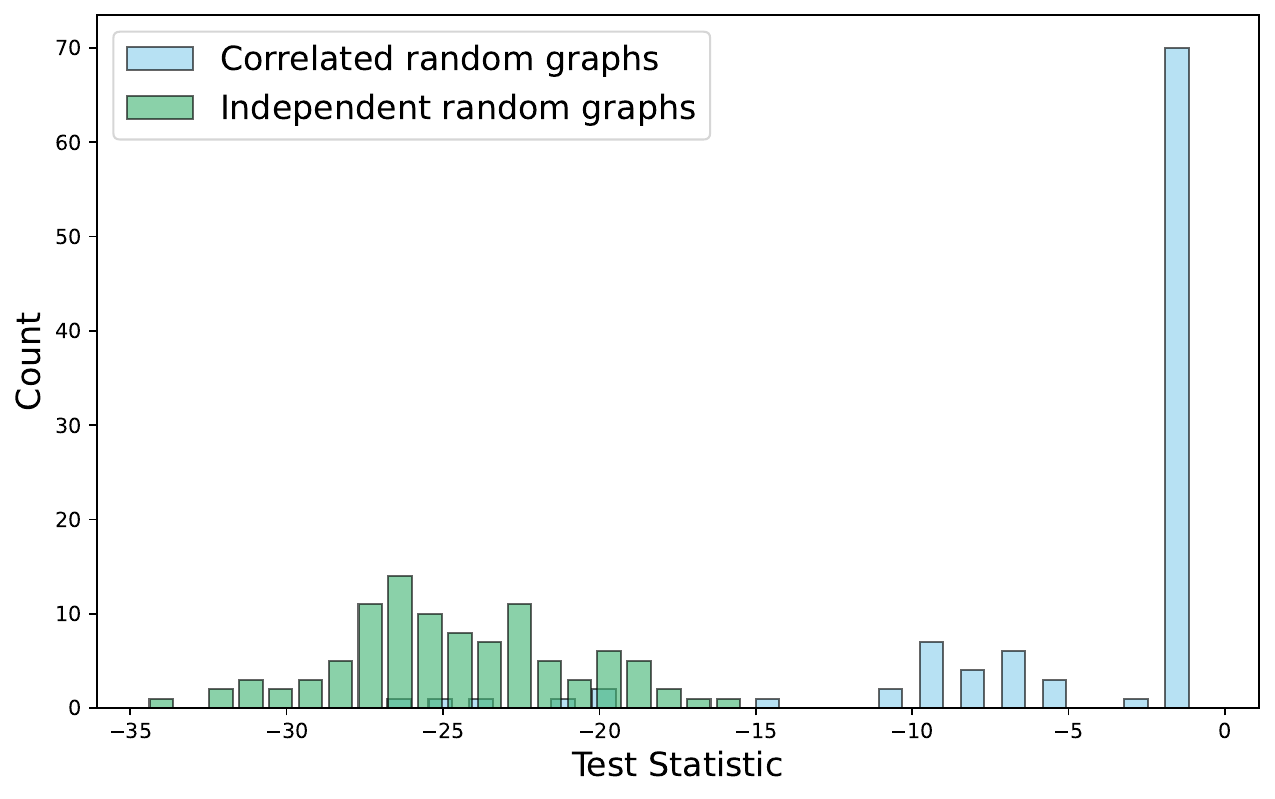}
    \caption{The histogram of the approximate test statistic $\sum_{e\in \binom{S_0}{2}} \beta_e\pth{\maH_{\pi_0}^f}$ in Algorithm~\ref{alg:clique-based} over 100 pairs of graphs, where the blue one represents the correlated Gaussian Wigner model, and the green one represents the independent graphs.}
    \label{fig: histogram}
\end{figure}

In order to compare our test statistic under different settings, we plot the Receiver Operating Characteristic (ROC) curves by varying the detection threshold and plotting the Type II error against the Type I error. We also compute the area under the curve (AUC), which can be interpreted as the probability that the  test statistic  is larger for a pair of correlated graphs than a pair of independent graphs. 

In Figure~\ref{fig:ROC-s}, for each plot, we fix $n = 50, \rho = 0.98, K_1 = 4, K_2 = 3, N_1 = 10000, N_2 = 500, \epsilon = 0.01$, and vary $s\in \sth{10, 20, 30, 40, 50}$, with $m = \left\lfloor\frac{(1-\epsilon)s^2}{n}\right\rfloor$.
We observe that as $s$ increases, the ROC curve is moving toward the upper left corner, and the AUC increases from $0.52$ to 1, indicating an improvement in the performance of our test statistic. Indeed, by Lemma~\ref{lem:hypergro-def}, the cardinality of common set increases as $s$ increase, strengthening the signal and facilitating correlation detection.

\begin{figure}[ht]
    \centering
    \includegraphics[width=0.5\linewidth]{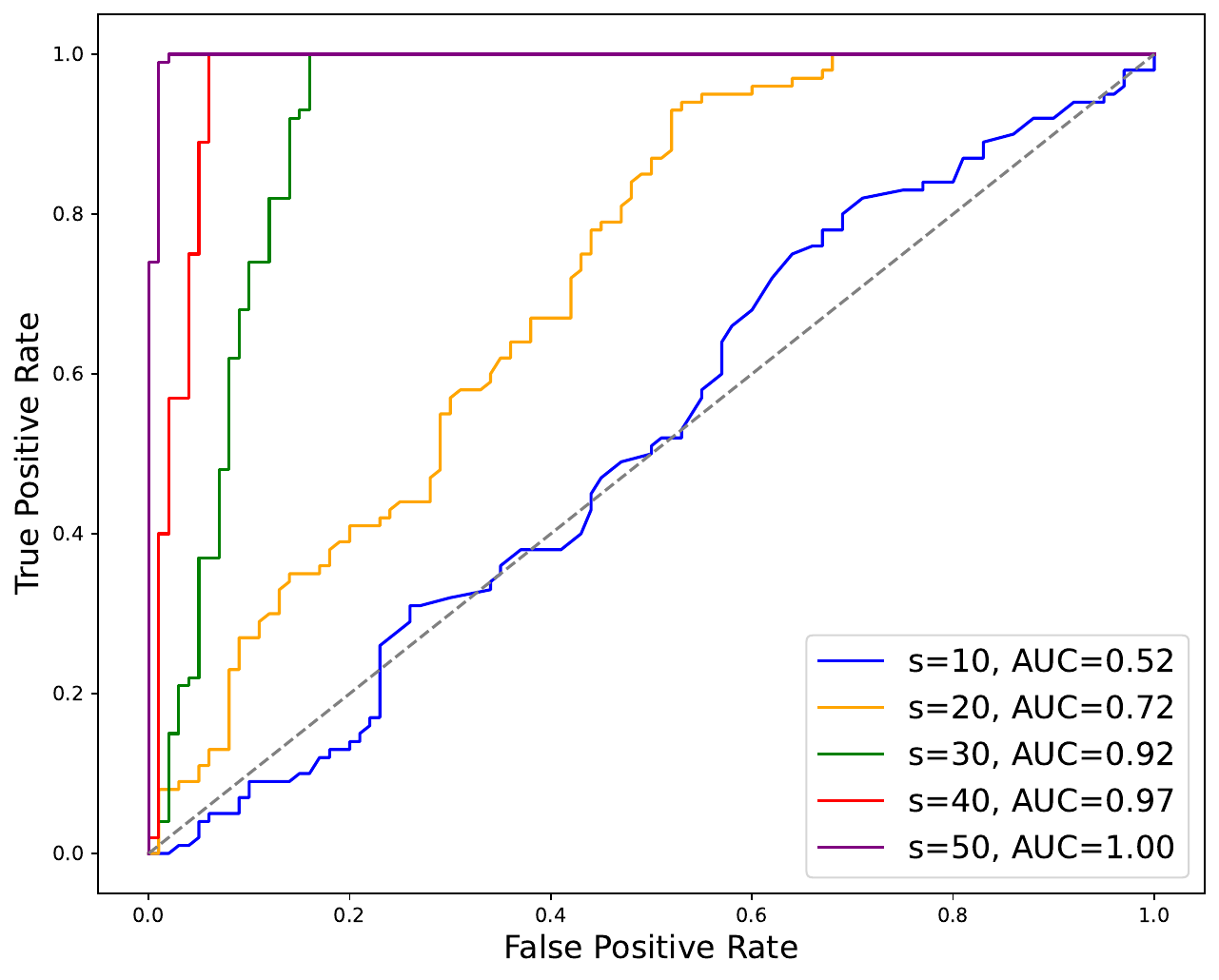}
    \caption{Comparison for the ROC curve of the approximate test statistic for different sample size $s$.}
    \label{fig:ROC-s}
\end{figure}

In Figure~\ref{fig:ROC-rho}, for each plot, we fix $n = 50, s=40, K_1 = 4, K_2 = 3, N_1 = 10000, N_2 = 500, \epsilon = 0.01$, and vary $\rho\in \sth{0.95, 0.96, 0.97,0.98,0.99}$, with $m = \left\lfloor\frac{(1-\epsilon)s^2}{n}\right\rfloor=31$.
We observe that as $\rho$ increases, the ROC curve is moving toward the upper left corner, and the AUC increases from $0.55$ to 1, indicating an improvement in the performance of our test statistic as the correlation strengthens. It turns out that correlation detection improves as $s$ and $\rho$ increase.

We also compare our method with the classical Graph Edit Distance (GED) ~\cite{sanfeliu1983distance},  a widely used graph similarity measure. When $n=50, s = 30$, and $\epsilon = 0.01$, the AUC values for the GED-based test at $\rho = 0.98, 1-10^{-6},1-10^{-7}$ are $0.53,0.73$, and $0.88$, respectively. In contrast, our algorithm yields significantly higher AUCs of $\rho$ are $0.92,1,1$ under the same settings. These results demonstrate the superior performance of our method in detecting correlation in the Gaussian Wigner model. We provide some additional experiments in Appendix~\ref{apdsec:additional-exp}.

\begin{figure}[ht]
    \centering
    \includegraphics[width=0.5\linewidth]{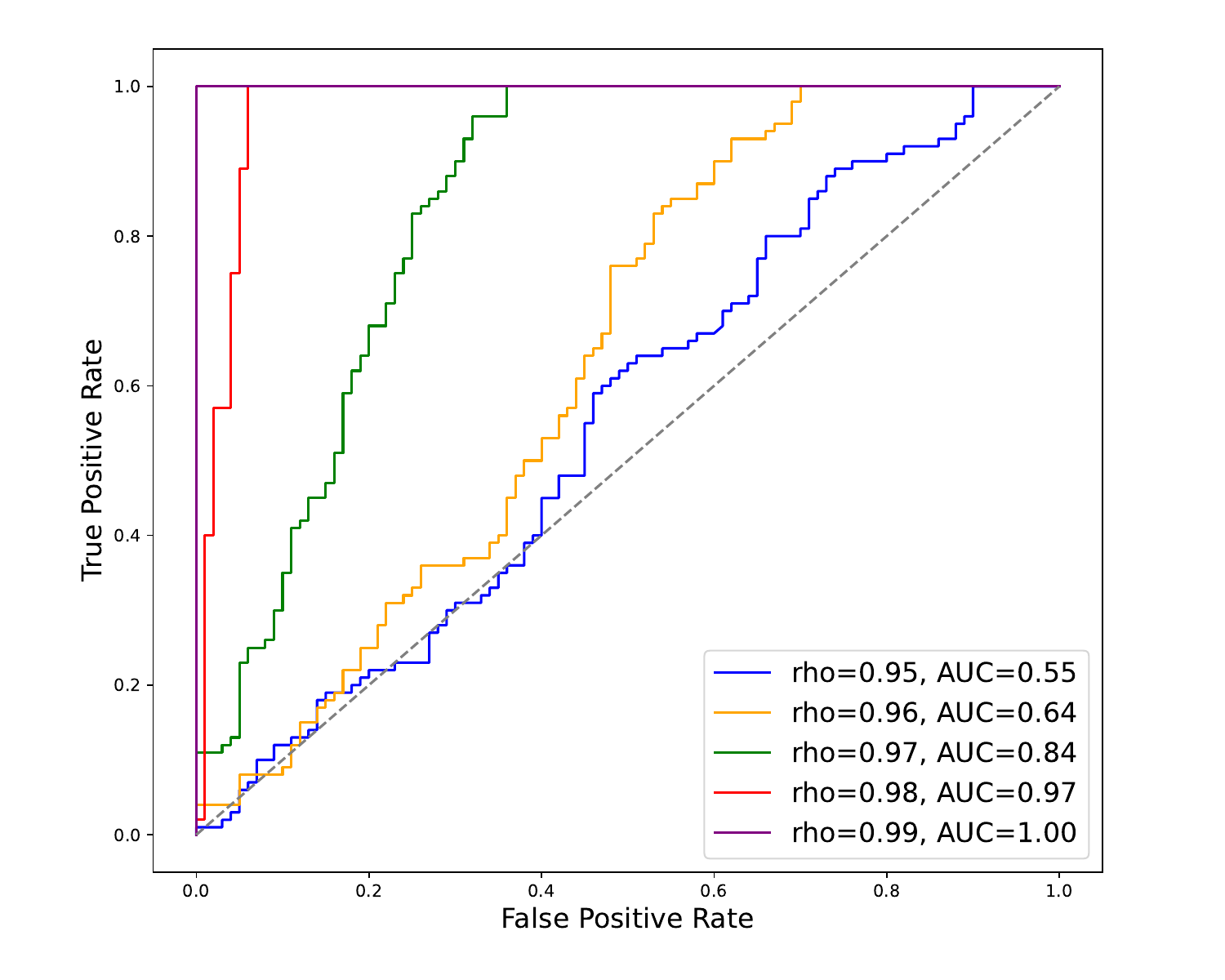}
    \caption{Comparison for the ROC curve of the approximate test statistic for different correlation coefficients $\rho$.}
    \label{fig:ROC-rho}
\end{figure}

\section{Future Directions and Discussions}\label{sec:future-direction}
This paper focuses on detecting correlation in the Gaussian Wigner model by sampling two induced subgraphs from the original graphs.
We determine the optimal rate on the sample size for correlation detection. In comparison to detection problem on the fully correlated Gaussian Wigner model, the additional challenge arises from  partial correlation when sampling subgraphs. We provide a detailed analysis of the \emph{core set} when using the conditional second moment method to derive the impossibility results. We find that the conditional second moment can be reduced to the second moment on the \emph{core set}. Additionally, we propose an efficient approximate algorithm for correlation detection based on the clique mapping technique and an iterative approach.
There are many problems to be further investigated:
\begin{itemize}
    \item \emph{Extension to \ER model.} Most results in this paper can be extended to the \ER model. The key difference lies in the additional parameter $p$ 
 controlling the edge connection probability. For the possibility results, the estimator is similar to~\eqref{eq:test-statistic}, with the bivariate function $f$ selected via MLE under the \ER model. For the impossibility results, the reduction procedure provides tight bounds when $p = n^{-\Omega(1)}$, and a more delicate event is required for the conditional second moment analysis when $p = n^{-o(1)}$, which is similar to Proposition~\ref{prop:lwbd-dense}.
 \item \emph{Theoretical analysis of the efficient algorithm.} We have shown that the Algorithm~\ref{alg:clique-based} performs well on synthetic data, while the theoretical guarantee remains an open problem. This guarantee can serve as an upper bound for the existence of a polynomial-time algorithm. Moreover, since the tree-counting-based method shows strong performance in the \ER model, it would be interesting to investigate whether it remains effective in Gaussian networks.
    \item \emph{Computational hardness.} The low-degree conjecture has recently provided evidence of the computational hardness on related problems (see, e.g., \cite{hopkins2018statistical,kunisky2019notes}). It is of interest to investigate the computational hardness conditions with respect to the sample size for the correlation detection problem using the low-degree conjecture.
    \item \emph{Other graph models.} The sample complexity for correlation detection remains unknown for many models (e.g., the stochastic block model, the graphon model). A natural next step is to explore whether our results can be extended to more general settings.
\end{itemize}
\bibliographystyle{alpha}
\bibliography{gaussian}

\newcommand{\etalchar}[1]{$^{#1}$}
\begin{thebibliography}{WWXY22}

\bibitem[ABT24]{araya2024seeded}
Ernesto Araya, Guillaume Braun, and Hemant Tyagi.
\newblock Seeded graph matching for the correlated gaussian wigner model via
  the projected power method.
\newblock {\em Journal of Machine Learning Research}, 25(5):1--43, 2024.

\bibitem[AH24]{ameen2024robust}
Taha Ameen and Bruce Hajek.
\newblock Robust graph matching when nodes are corrupt.
\newblock In {\em Proceedings of the 41st International Conference on Machine
  Learning}, volume 235, pages 1276--1305. PMLR, 2024.

\bibitem[AH25]{ameen2025detecting}
Taha Ameen and Bruce Hajek.
\newblock Detecting correlation between multiple unlabeled gaussian networks.
\newblock {\em arXiv preprint arXiv:2504.16279}, 2025.

\bibitem[BBM05]{berg2005shape}
Alexander~C Berg, Tamara~L Berg, and Jitendra Malik.
\newblock Shape matching and object recognition using low distortion
  correspondences.
\newblock In {\em 2005 IEEE computer society conference on computer vision and
  pattern recognition (CVPR'05)}, volume~1, pages 26--33. IEEE, 2005.

\bibitem[BCL{\etalchar{+}}19]{barak2019nearly}
Boaz Barak, Chi-Ning Chou, Zhixian Lei, Tselil Schramm, and Yueqi Sheng.
\newblock ({N}early) efficient algorithms for the graph matching problem on
  correlated random graphs.
\newblock {\em Advances in Neural Information Processing Systems}, 32, 2019.

\bibitem[BES80]{babai1980random}
L{\'a}szl{\'o} Babai, Paul Erd{\H{o}}s, and Stanley~M Selkow.
\newblock Random graph isomorphism.
\newblock {\em SIaM Journal on computing}, 9(3):628--635, 1980.

\bibitem[Bol82]{bollobas1982distinguishing}
B{\'e}la Bollob{\'a}s.
\newblock Distinguishing vertices of random graphs.
\newblock In {\em North-Holland Mathematics Studies}, volume~62, pages 33--49.
  Elsevier, 1982.

\bibitem[CCLS07]{caetano2007learning}
Tiberio~S. Caetano, Li~Cheng, Quoc~V. Le, and Alex~J. Smola.
\newblock Learning graph matching.
\newblock In {\em 2007 IEEE 11th International Conference on Computer Vision},
  pages 1--8, 2007.

\bibitem[CDGL24]{chen2024computational}
Guanyi Chen, Jian Ding, Shuyang Gong, and Zhangsong Li.
\newblock A computational transition for detecting correlated stochastic block
  models by low-degree polynomials.
\newblock {\em arXiv preprint arXiv:2409.00966}, 2024.

\bibitem[CDGL25]{chen2025detecting}
Guanyi Chen, Jian Ding, Shuyang Gong, and Zhangsong Li.
\newblock Detecting correlation efficiently in stochastic block models:
  breaking otter's threshold by counting decorated trees.
\newblock {\em arXiv preprint arXiv:2503.06464}, 2025.

\bibitem[CK16]{cullina2016improved}
Daniel Cullina and Negar Kiyavash.
\newblock Improved achievability and converse bounds for
  {\uppercase{e}}rd{\H{o}}s-{\uppercase{r}}{\'e}nyi graph matching.
\newblock {\em ACM SIGMETRICS performance evaluation review}, 44(1):63--72,
  2016.

\bibitem[CK17]{cullina2017exact}
Daniel Cullina and Negar Kiyavash.
\newblock Exact alignment recovery for correlated
  {\uppercase{e}}rd{\H{o}}s-{\uppercase{r}}{\'e}nyi graphs.
\newblock {\em arXiv preprint arXiv:1711.06783}, 2017.

\bibitem[Dav79]{davis1979circulant}
P.J. Davis.
\newblock {\em Circulant Matrices}.
\newblock Wiley, 1979.

\bibitem[DCKG19]{dai2019analysis}
Osman~Emre Dai, Daniel Cullina, Negar Kiyavash, and Matthias Grossglauser.
\newblock Analysis of a canonical labeling algorithm for the alignment of
  correlated {\uppercase{e}}rd{\H{o}}s-{\uppercase{r}}{\'e}nyi graphs.
\newblock {\em Proceedings of the ACM on Measurement and Analysis of Computing
  Systems}, 3(2):1--25, 2019.

\bibitem[DD23]{ding2023matching}
Jian Ding and Hang Du.
\newblock Matching recovery threshold for correlated random graphs.
\newblock {\em The Annals of Statistics}, 51(4):1718--1743, 2023.

\bibitem[DDL23]{ding2023low}
Jian Ding, Hang Du, and Zhangsong Li.
\newblock Low-degree hardness of detection for correlated
  {\uppercase{e}}rd{\H{o}}s-{\uppercase{r}}{\'e}nyi graphs.
\newblock {\em arXiv preprint arXiv:2311.15931}, 2023.

\bibitem[DFW23]{ding2023efficiently}
Jian Ding, Yumou Fei, and Yuanzheng Wang.
\newblock Efficiently matching random inhomogeneous graphs via degree profiles.
\newblock {\em arXiv preprint arXiv:2310.10441}, 2023.

\bibitem[DL23]{ding2023polynomial}
Jian Ding and Zhangsong Li.
\newblock A polynomial-time iterative algorithm for random graph matching with
  non-vanishing correlation.
\newblock {\em arXiv preprint arXiv:2306.00266}, 2023.

\bibitem[DL24]{ding2024polynomial}
Jian Ding and Zhangsong Li.
\newblock A polynomial time iterative algorithm for matching gaussian matrices
  with non-vanishing correlation.
\newblock {\em Foundations of Computational Mathematics}, pages 1--58, 2024.

\bibitem[DMWX21]{ding2021efficient}
Jian Ding, Zongming Ma, Yihong Wu, and Jiaming Xu.
\newblock Efficient random graph matching via degree profiles.
\newblock {\em Probability Theory and Related Fields}, 179:29--115, 2021.

\bibitem[ER59]{paul1959random}
Paul Erd{\H{o}}s and Alfr{\'e}d R{\'e}nyi.
\newblock On random graphs {I}.
\newblock {\em Publicationes Mathematicae (Debrecen)}, 6:290--297, 1959.

\bibitem[FF79]{freeman1979networkers}
S.C. Freeman and L.C. Freeman.
\newblock {\em The Networkers Network: A Study of the Impact of a New
  Communications Medium on Sociometric Structure}.
\newblock Social sciences research reports. School of Social Sciences
  University of California, 1979.

\bibitem[FMWX23]{fan2019spectral}
Zhou Fan, Cheng Mao, Yihong Wu, and Jiaming Xu.
\newblock Spectral graph matching and regularized quadratic relaxations {I}:
  The gaussian model.
\newblock {\em Foundations of Computational Mathematics}, 23(5):1511--1565,
  2023.

\bibitem[Gho21]{ghosh2021exponential}
Malay Ghosh.
\newblock Exponential tail bounds for chisquared random variables.
\newblock {\em Journal of Statistical Theory and Practice}, 15(2):35, 2021.

\bibitem[GM20]{ganassali2020tree}
Luca Ganassali and Laurent Massouli{\'e}.
\newblock From tree matching to sparse graph alignment.
\newblock In {\em Conference on Learning Theory}, pages 1633--1665. PMLR, 2020.

\bibitem[GMS24]{ganassali2024statistical}
Luca Ganassali, Laurent Massouli{\'e}, and Guilhem Semerjian.
\newblock Statistical limits of correlation detection in trees.
\newblock {\em The Annals of Applied Probability}, 34(4):3701--3734, 2024.

\bibitem[HL13]{hu2013survey}
Pili Hu and Wing~Cheong Lau.
\newblock A survey and taxonomy of graph sampling.
\newblock {\em arXiv preprint arXiv:1308.5865}, 2013.

\bibitem[HM23]{hall2023partial}
Georgina Hall and Laurent Massouli{\'e}.
\newblock Partial recovery in the graph alignment problem.
\newblock {\em Operations Research}, 71(1):259--272, 2023.

\bibitem[HNM05]{haghighi2005robust}
Aria Haghighi, Andrew~Y Ng, and Christopher~D Manning.
\newblock Robust textual inference via graph matching.
\newblock In {\em Proceedings of Human Language Technology Conference and
  Conference on Empirical Methods in Natural Language Processing}, pages
  387--394, 2005.

\bibitem[Hoe94]{hoeffding1994probability}
Wassily Hoeffding.
\newblock Probability inequalities for sums of bounded random variables.
\newblock {\em The collected works of Wassily Hoeffding}, pages 409--426, 1994.

\bibitem[Hop18]{hopkins2018statistical}
Samuel Hopkins.
\newblock {\em Statistical inference and the sum of squares method}.
\newblock PhD thesis, Cornell University, 2018.

\bibitem[HR07]{hughes2007lexical}
Thad Hughes and Daniel Ramage.
\newblock Lexical semantic relatedness with random graph walks.
\newblock In {\em Proceedings of the 2007 joint conference on empirical methods
  in natural language processing and computational natural language learning
  (EMNLP-CoNLL)}, pages 581--589, 2007.

\bibitem[HS17]{hopkins2017efficient}
Samuel~B Hopkins and David Steurer.
\newblock Efficient bayesian estimation from few samples: community detection
  and related problems.
\newblock In {\em 2017 IEEE 58th Annual Symposium on Foundations of Computer
  Science (FOCS)}, pages 379--390. IEEE, 2017.

\bibitem[HSY24]{huang2024information}
Dong Huang, Xianwen Song, and Pengkun Yang.
\newblock Information-theoretic thresholds for the alignments of partially
  correlated graphs.
\newblock In {\em Conference on Learning Theory}, pages 2494--2518. PMLR, 2024.

\bibitem[HW71]{hanson1971bound}
David~Lee Hanson and Farroll~Tim Wright.
\newblock A bound on tail probabilities for quadratic forms in independent
  random variables.
\newblock {\em The Annals of Mathematical Statistics}, 42(3):1079--1083, 1971.

\bibitem[JRA{\etalchar{+}}22]{jiang2022graph}
Zheheng Jiang, Hossein Rahmani, Plamen Angelov, Sue Black, and Bryan~M
  Williams.
\newblock Graph-context attention networks for size-varied deep graph matching.
\newblock In {\em Proceedings of the IEEE/CVF Conference on Computer Vision and
  Pattern Recognition}, pages 2343--2352, 2022.

\bibitem[KWB19]{kunisky2019notes}
Dmitriy Kunisky, Alexander~S Wein, and Afonso~S Bandeira.
\newblock Notes on computational hardness of hypothesis testing: Predictions
  using the low-degree likelihood ratio.
\newblock In {\em ISAAC Congress (International Society for Analysis, its
  Applications and Computation)}, pages 1--50. Springer, 2019.

\bibitem[LF06]{leskovec2006sampling}
Jure Leskovec and Christos Faloutsos.
\newblock Sampling from large graphs.
\newblock In {\em Proceedings of the 12th ACM SIGKDD international conference
  on Knowledge discovery and data mining}, pages 631--636, 2006.

\bibitem[LR13]{livi2013graph}
Lorenzo Livi and Antonello Rizzi.
\newblock The graph matching problem.
\newblock {\em Pattern Analysis and Applications}, 16:253--283, 2013.

\bibitem[MHK{\etalchar{+}}08]{mateus2008articulated}
Diana Mateus, Radu Horaud, David Knossow, Fabio Cuzzolin, and Edmond Boyer.
\newblock Articulated shape matching using laplacian eigenfunctions and
  unsupervised point registration.
\newblock In {\em 2008 IEEE Conference on Computer Vision and Pattern
  Recognition}, pages 1--8. IEEE, 2008.

\bibitem[MRT23]{mao2023exact}
Cheng Mao, Mark Rudelson, and Konstantin Tikhomirov.
\newblock Exact matching of random graphs with constant correlation.
\newblock {\em Probability Theory and Related Fields}, 186(1-2):327--389, 2023.

\bibitem[MS24]{muratori2024faster}
Andrea Muratori and Guilhem Semerjian.
\newblock Faster algorithms for the alignment of sparse correlated
  {\uppercase{e}}rd{\H{o}}s-{\uppercase{r}}{\'e}nyi random graphs.
\newblock {\em arXiv preprint arXiv:2405.08421}, 2024.

\bibitem[MU05]{mitzenmacher_probability_2005}
Michael Mitzenmacher and Eli Upfal.
\newblock {\em Probability and computing: an introduction to randomized
  algorithms and probabilistic analysis}.
\newblock Cambridge University Press, 2005.

\bibitem[MWXY23]{mao2023random}
Cheng Mao, Yihong Wu, Jiaming Xu, and Sophie~H Yu.
\newblock Random graph matching at otter’s threshold via counting
  chandeliers.
\newblock In {\em Proceedings of the 55th Annual ACM Symposium on Theory of
  Computing}, pages 1345--1356, 2023.

\bibitem[MWXY24]{mao2024testing}
Cheng Mao, Yihong Wu, Jiaming Xu, and Sophie~H Yu.
\newblock Testing network correlation efficiently via counting trees.
\newblock {\em The Annals of Statistics}, 52(6):2483--2505, 2024.

\bibitem[MX20]{mossel2020seeded}
Elchanan Mossel and Jiaming Xu.
\newblock Seeded graph matching via large neighborhood statistics.
\newblock {\em Random Structures \& Algorithms}, 57(3):570--611, 2020.

\bibitem[Ott48]{otter1948number}
Richard Otter.
\newblock The number of trees.
\newblock {\em Annals of Mathematics}, pages 583--599, 1948.

\bibitem[PDK11]{papagelis2011sampling}
Manos Papagelis, Gautam Das, and Nick Koudas.
\newblock Sampling online social networks.
\newblock {\em IEEE Transactions on knowledge and data engineering},
  25(3):662--676, 2011.

\bibitem[PSSZ22]{piccioli2022aligning}
Giovanni Piccioli, Guilhem Semerjian, Gabriele Sicuro, and Lenka Zdeborov{\'a}.
\newblock Aligning random graphs with a sub-tree similarity message-passing
  algorithm.
\newblock {\em Journal of Statistical Mechanics: Theory and Experiment},
  2022(6):063401, 2022.

\bibitem[PW25]{polyanskiy2025information}
Yury Polyanskiy and Yihong Wu.
\newblock {\em Information theory: From coding to learning}.
\newblock Cambridge university press, 2025.

\bibitem[RNW24]{ratnayaka2024optimal}
Gathika Ratnayaka, James Nichols, and Qing Wang.
\newblock Optimal partial graph matching.
\newblock {\em arXiv preprint arXiv:2410.16718}, 2024.

\bibitem[RS23]{racz2023matching}
Mikl{\'o}s~Z R{\'a}cz and Anirudh Sridhar.
\newblock Matching correlated inhomogeneous random graphs using the k-core
  estimator.
\newblock In {\em 2023 IEEE International Symposium on Information Theory
  (ISIT)}, pages 2499--2504. IEEE, 2023.

\bibitem[SF83]{sanfeliu1983distance}
Alberto Sanfeliu and King-Sun Fu.
\newblock A distance measure between attributed relational graphs for pattern
  recognition.
\newblock {\em IEEE Transactions on Systems, Man, and Cybernetics},
  SMC-13(3):353--362, 1983.

\bibitem[SNK18]{sharma2018solving}
Charu Sharma, Deepak Nathani, and Manohar Kaul.
\newblock Solving partial assignment problems using random clique complexes.
\newblock In {\em International Conference on Machine Learning}, pages
  4586--4595. PMLR, 2018.

\bibitem[SWM05]{stumpf2005subnets}
Michael~PH Stumpf, Carsten Wiuf, and Robert~M May.
\newblock Subnets of scale-free networks are not scale-free: sampling
  properties of networks.
\newblock {\em Proceedings of the National Academy of Sciences},
  102(12):4221--4224, 2005.

\bibitem[SXB08]{singh2008global}
Rohit Singh, Jinbo Xu, and Bonnie Berger.
\newblock Global alignment of multiple protein interaction networks with
  application to functional orthology detection.
\newblock {\em Proceedings of the National Academy of Sciences},
  105(35):12763--12768, 2008.

\bibitem[Tsy09]{tsybakov2009introduction}
A.~B. Tsybakov.
\newblock {\em Introduction to Nonparametric Estimation}.
\newblock Springer Verlag, 2009.

\bibitem[VCL{\etalchar{+}}15]{vogelstein2015fast}
Joshua~T Vogelstein, John~M Conroy, Vince Lyzinski, Louis~J Podrazik, Steven~G
  Kratzer, Eric~T Harley, Donniell~E Fishkind, R~Jacob Vogelstein, and Carey~E
  Priebe.
\newblock Fast approximate quadratic programming for graph matching.
\newblock {\em PLOS ONE}, 10(4):1--17, 2015.

\bibitem[WCA{\etalchar{+}}16]{wu2016evaluation}
Yanhong Wu, Nan Cao, Daniel Archambault, Qiaomu Shen, Huamin Qu, and Weiwei
  Cui.
\newblock Evaluation of graph sampling: A visualization perspective.
\newblock {\em IEEE transactions on visualization and computer graphics},
  23(1):401--410, 2016.

\bibitem[WGJ{\etalchar{+}}23]{wang2023deep}
Runzhong Wang, Ziao Guo, Shaofei Jiang, Xiaokang Yang, and Junchi Yan.
\newblock Deep learning of partial graph matching via differentiable top-k.
\newblock In {\em Proceedings of the IEEE/CVF Conference on Computer Vision and
  Pattern Recognition}, pages 6272--6281, 2023.

\bibitem[WWXY22]{wang2022random}
Haoyu Wang, Yihong Wu, Jiaming Xu, and Israel Yolou.
\newblock Random graph matching in geometric models: the case of complete
  graphs.
\newblock In {\em Conference on Learning Theory}, pages 3441--3488. PMLR, 2022.

\bibitem[WXY22]{wu2022settling}
Yihong Wu, Jiaming Xu, and Sophie~H Yu.
\newblock Settling the sharp reconstruction thresholds of random graph
  matching.
\newblock {\em IEEE Transactions on Information Theory}, 68(8):5391--5417,
  2022.

\bibitem[WXY23]{wu2023testing}
Yihong Wu, Jiaming Xu, and Sophie~H Yu.
\newblock Testing correlation of unlabeled random graphs.
\newblock {\em The Annals of Applied Probability}, 33(4):2519--2558, 2023.

\bibitem[YG13]{yartseva2013performance}
Lyudmila Yartseva and Matthias Grossglauser.
\newblock On the performance of percolation graph matching.
\newblock In {\em Proceedings of the first ACM conference on Online social
  networks}, pages 119--130, 2013.

\end{thebibliography}
\appendix

\section{Proof of Theorem~\ref{thm:main-thm}}\label{apd:main-thm}

For the possibility results, by Propositions~\ref{prop: upbd-overlap} and~\ref{prop: upbd-least-square}, if \begin{align*}
    s^2 \ge \begin{cases}
        \frac{C_1 n\log n}{\rho^2}&0<\rho\le 1-e^{-6}\\
        C_2\pth{\frac{n\log n}{\log\pth{1/(1-\rho)}}\vee n}&1-e^{-6}<\rho<1
    \end{cases},
\end{align*}
then $\TV\pth{\maP,\maQ}\le 0.1$. Furthermore, if $s^2=\omega(n)$, then $\TV(\maP,\maQ) = o(1)$. Since $\frac{\log\pth{1/(1-\rho^2)}}{\rho^2}\le \frac{\rho^2/(1-\rho^2)}{\rho^2}=\frac{1}{1-\rho^2}\le \frac{1}{1-\pth{1-e^{-6}}^2}$ for any $0<\rho\le 1-e^{-6}$,  we obtain that $\TV\pth{\maP,\maQ} \le 0.1$ when $s^2 \ge \frac{C_1}{1-\pth{1-e^{-6}}^2}\cdot \frac{n\log n}{\log\pth{1/(1-\rho^2)}}$.
Since $\frac{\log\pth{1/(1-\rho^2)}}{\log\pth{1/(1-\rho)}} = 1+\frac{\log\pth{1/(1+\rho)}}{\log\pth{1/(1-\rho)}}\le 2$ for any $1-e^{-6}<\rho<1$, it follows that when $s^2 \ge 2C_2\pth{\frac{n\log n}{\log\pth{1/(1-\rho)}}\vee n}$, $\TV\pth{\maP,\maQ}\le 0.1$. Let $\overline{C} = \frac{C_1}{1-(1-e^{-6})^2}\vee 2C_2$. Then, for $s^2\ge \overline{C}\pth{\frac{n\log n}{\log\pth{1/(1-\rho^2)}}\vee n}$, we have $\TV\pth{\maP,\maQ}\le 0.1$.

For the impossibility results, by Propositions~\ref{prop:lwbd-sparse} and~\ref{prop:lwbd-dense}, if $s^2 \le \frac{n\log n}{8\log\pth{1/(1-\rho^2)}}$, then $\TV\pth{\maP,\maQ} = o(1)$. According to the concentration inequality~\eqref{eq:concentration_for_hyper_2} for the Hypergeometric distribution in Lemma~\ref{lem:Hypergeometric_distribution_prop}, there exists a constant $\underline{C}\le \frac{1}{8}$ such that, when $s^2\le \underline{C} n$, we have $\prob{|S_{\pi^*}|\ge 1}\le 0.1$, implying $\prob{|S_{\pi^*}|=0}\ge 0.9$ and thus $\TV\pth{\maP,\maQ}\le 0.1$.
Additionally, when $s^2 = o(n)$, we have $\prob{|S_{\pi^*}|=0}=1-o(1)$, which implies $\TV\pth{\maP,\maQ} = o(1)$.
Therefore, if $s^2 \le \underline{C} \pth{\frac{n\log n}{\log\pth{1/(1-\rho^2)}}\vee n}$, then $\TV\pth{\maP,\maQ}\le 0.1$. Moreover, if $s^2 \le \underline{C}\pth{\frac{n\log n}{\log\pth{1/(1-\rho^2)}}\vee n}$ or $s^2=o(n)$, we have $\TV\pth{\maP,\maQ} = o(1)$. This concludes the proof of Theorem~\ref{thm:main-thm}.

\section{Proof of Propositions}
\subsection{Proof of Proposition~\ref{prop: upbd-overlap}}\label{apd:proof-upbd-overlap}
We first upper bound $\maQ\pth{\maT\ge \tau}$ under the null hypothesis $\maH_0$ by the Chernoff bound and union bound.  
For any $X,Y\overset{\mathrm{i.i.d.}}{\sim} \maN(0,1)$ and $\lambda\in \pth{0,1}$, we have \begin{align}
    \nonumber \expect{\exp\pth{\lambda XY}} =&~\int \int \frac{1}{2\pi}\exp\pth{\lambda xy}\exp\pth{-\frac{x^2+y^2}{2}}\,dxdy\\\nonumber =&~\int\int\frac{1}{2\pi} \exp\pth{-\frac{1}{2}\pth{x-\lambda y}^2}\exp\pth{-\frac{1}{2}(1-\lambda^2)y^2}\,dxdy\\\label{eq:mgf-weak}
    =&~\int\int\frac{1}{2\pi}\exp\pth{-\frac{z^2}{2}}\exp\pth{-\frac{1}{2}(1-\lambda^2)y^2}\,dzdy = \frac{1}{\sqrt{1-\lambda^2}}. 
\end{align}
Let $\lambda = \frac{\rho}{2}$. Recall that $\maS_{s,m}$ denotes the set of injective mappings $\pi:S\subseteq V(G_1)\mapsto V(G_2)$ with $|S| = m$. For any $\pi\in\maS_{s,m}$, $\en\pth{\maH_\pi^f}\sim \sum_{i=1}^{\binom{m}{2}} A_i B_i$, where $(A_i,B_i)$ are independent and identically distributed pairs of standard normals with correlation coefficient $\rho$.
Then, by the Chernoff bound, 
\begin{align}
    \label{eq:lambda-1}\maQ\qth{\en\pth{\maH_{\pi}^f}\ge \tau}\le&~ \exp\pth{-\lambda \tau} \expect{\exp\pth{\lambda \en\pth{\maH_{\pi}^f}}}\\\nonumber=&~\exp\pth{-\lambda\tau }\expect{\prod_{i=1}^m \exp\pth{\lambda A_i B_i}} \\\nonumber 
    \overset{\mathrm{(a)}}{\le}&~\exp\pth{-\lambda \binom{m}{2}\frac{\rho}{2}-\frac{1}{2}\binom{m}{2}\log\pth{1-\lambda^2}}\\\label{eq:chernoff-upbd-overlap}
    =&~\exp\pth{-\binom{m}{2}\pth{\frac{\rho^2}{4}+\frac{1}{2}\log\pth{1-\frac{\rho^2}{4}}}}\overset{\mathrm{(b)}}{\le} \exp\pth{-\frac{1}{12}\binom{m}{2}\rho^2},
\end{align}
where $\mathrm{(a)}$ is because $\expect{\lambda A_i B_i} = \frac{1}{\sqrt{1-\lambda^2}}$ for any $1\le i\le \binom{m}{2}$; $\mathrm{(b)}$ follows from $\log(1-x)\ge -\frac{1}{3}x$ for $x=\frac{\rho^2}{4}\in \qth{0,\frac{1}{4}}$.
Applying the union bound, we obtain that \begin{align*}
    \maQ\pth{\maT \ge \tau}\le  |\maS_{s,m}| \maQ \qth{\en\pth{\maH_{\pi}^f}\ge \tau}
    &\overset{\mathrm{(a)}}{\le} \binom{s}{m}^2 m! \exp\pth{-\frac{1}{12}\binom{m}{2}\rho^2} \\&\overset{\mathrm{(b)}}{\le}\exp\pth{m\log\pth{\frac{en}{1-\epsilon}}-\frac{1}{12}\binom{m}{2}\rho^2}, 
\end{align*}
where $\mathrm{(a)}$ is because $|\maS_{s,m}| = \tbinom{s}{m}^2 m!$ and~\eqref{eq:chernoff-upbd-overlap}; $\mathrm{(b)}$ is because $\binom{s}{m}m!\le s^m$, $\binom{s}{m}\le \pth{\frac{e\cdot s}{m}}^m$ and $m=\frac{(1-\epsilon)s^2}{n}$.
Consequently, when $m-1\ge \frac{24(1+\epsilon)\log\pth{\frac{en}{1-\epsilon}}}{\rho^2}$, we have $\maQ\pth{\maT\ge \tau}\le \exp\pth{-\epsilon m\log\pth{\frac{en}{1+\epsilon}}} = o(1)$.

We then upper bound $\maP\pth{\maT<\tau}$ under the alternative hypothesis $\maH_1$.
We note that \begin{align}
        \nonumber \maP\pth{\maT<\tau}\overset{\mathrm{(a)}}{\le}&~ \maP\pth{|S_{\pi^*}|<m}+\maP\pth{\maT<\tau,|S_{\pi^*}|\ge m}\\\nonumber\overset{\mathrm{(b)}}{\le}&~ \maP\pth{|S_{\pi^*}|<m}+\maP\pth{\maT<\tau\big|\, |S_{\pi^*}|\ge m}\\\label{eq:upbd-mat-tau}\overset{\mathrm{(c)}}{\le}&~ \maP\pth{|S_{\pi^*}|<m}+\maP\pth{\en\pth{\maH_{\pi^*_m}^f}<\tau\big|\,|S_{\pi^*}|\ge m}\\\label{eq:upbd-concentration-gau}
        \overset{\mathrm{(d)}}{\le}&~\exp\pth{-\frac{\epsilon^2 s^2}{2n}}+\exp\pth{-\binom{m}{2}\frac{\rho^2}{4c_0^2}}+\exp\pth{-\binom{m}{2}\frac{\rho}{2c_0}},
    \end{align}
where $\mathrm{(a)}$ follows from~\eqref{eq:event_subseteq}; $\mathrm{(b)}$ is because $$\maP\pth{\maT<\tau,|S_{\pi^*}|\ge m}\le \frac{\maP\pth{\maT<\tau,|S_{\pi^*}|\ge m}}{\maP\pth{|S_{\pi^*}|\ge m}} = \maP\pth{\maT<\tau\big|\,|S_{\pi^*}|\ge m};$$
$\mathrm{(c)}$ is because under the event $|S_{\pi^*}|\ge m$, there exists $\pi_m^*\in \maS_{s,m}$ such that $\pi_m^* = \pi^*$ on its domain set $\mathrm{dom}(\pi_m^*)$; 
$\mathrm{(d)}$ uses the concentration~\eqref{eq:concentration_for_hyper_1} for Hypergeometric distribution and the Hanson-Wright inequality in Lemma~\ref{lem:Hanson-Wright} with $M_0 = I_{\binom{m}{2}}$ and $\delta = \exp\pth{-\binom{m}{2}\pth{\frac{\rho^2}{4c_0^2} \wedge \frac{\rho}{2c_0}}}$, where $c_0$ is the universal constant in Lemma~\ref{lem:Hanson-Wright}.
Consequently, we obtain that $\maP\pth{\maT<\tau} = o(1)$ when $m-1\ge \frac{24(1+\epsilon)\log\pth{\frac{en}{1-\epsilon}}}{\rho^2}$. Let $C_1 = 25$. Then, we have $\maP\pth{\maT<\tau}+\maQ\pth{\maT\ge \tau} = o(1)$ when $\frac{s^2}{n}-1\ge \frac{25\log n}{\rho^2}$ as $n$ becomes sufficiently large.

\subsection{Proof of Proposition~\ref{prop: upbd-least-square}}\label{apd:upbd-least-square}
We first upper bound $\maQ\pth{\maT\ge \tau}$ under the null hypothesis $\maH_0$. We note that for any $X,Y\overset{\mathrm{i.i.d.}}{\sim} \maN(0,1)$ and $\lambda>0$, we have \begin{align}
        \nonumber \expect{\exp\pth{ -\frac{\lambda}{2}(X-Y)^2}} =&~ \int\int \frac{1}{2\pi} \exp\pth{-\frac{\lambda}{2}(x-y)^2} \exp\pth{-\frac{1}{2}\pth{x^2+y^2}}\,dxdy \\\nonumber
        =&~\int\int\frac{1}{2\pi}\exp\pth{-\frac{\lambda+1}{2}\pth{x-\frac{\lambda}{\lambda+1}y}^2} \exp\pth{-\frac{2\lambda+1}{2(\lambda+1)} y^2}\,dxdy \\\label{eq:mgf-strong}
        =&~  \int\int \frac{1}{2\pi} \exp\pth{-\frac{\lambda+1}{2} z^2} \exp\pth{-\frac{2\lambda+1}{2(\lambda+1)} y^2}\,dydz =\frac{1}{\sqrt{1+2\lambda}}.
\end{align}
Let $\lambda = \frac{1}{4(1-\rho)}-\frac{1}{2}$. Then, we have $1+2\lambda = \frac{1}{2(1-\rho)}$.
Since $1-e^{-6}<\rho<1$, we also have $\lambda>0$.
Recall that $\maS_{s,m}$ denotes the set of injective mappings $\pi:S\subseteq V(G_1)\mapsto V(G_2)$ with $|S| = m$. For any $\pi\in\maS_{s,m}$, $\en\pth{\maH_\pi^f} \sim \sum_{i=1}^{\binom{m}{2}} -\frac{1}{2}\pth{A_i-B_i}^2$, where $(A_i,B_i)$ are independent and identically distributed pairs of standard normals with correlation coefficient $\rho$.
Then, by the Chernoff bound, \begin{align}
        \nonumber\maQ\pth{\en\pth{\maH_\pi^f}\ge \tau} \le&~ \exp\pth{-\lambda \tau}\expect{\exp\pth{\lambda\en\pth{\maH_\pi^f}}}\\
        \nonumber \overset{\mathrm{(a)}}{=}&~\exp\pth{\binom{m}{2}\pth{2(1-\rho)\lambda-\frac{1}{2}\log\pth{1+2\lambda}} }\\
        \nonumber =&~\exp\pth{\binom{m}{2}\pth{\frac{1}{2}-(1-\rho)-\frac{1}{2}\log\pth{\frac{1}{2(1-\rho)}}}}\\
        \label{eq:chernoff-upbd-least-square}\overset{\mathrm{(b)}}{\le}&~ \exp\pth{-\frac{\log(1/(1-\rho))}{3}\binom{m}{2}},
    \end{align}
    where $\mathrm{(a)}$ follows from~\eqref{eq:mgf-strong}; $\mathrm{(b)}$ is because $\rho>1-e^{-6}$ implies that $\frac{1}{2}-(1-\rho)-\frac{1}{2}\log\pth{\frac{1}{2(1-\rho)}}\le 1-\frac{1}{6}\log\pth{\frac{1}{1-\rho}}-\frac{1}{3}\log\pth{\frac{1}{1-\rho}}\le -\frac{1}{3}\log\pth{\frac{1}{1-\rho}}$.
Then, applying the union bound yields that \begin{align*}
        \maQ\pth{\maT\ge \tau}\le&~ |\maS_{s,m}| \maQ\qth{\en\pth{\maH_\pi^f}\ge \tau} \le\exp\pth{m\log\pth{\frac{en}{1-\epsilon}}-\frac{1}{3}\log\pth{\frac{1}{1-\rho}}\binom{m}{2}},
    \end{align*}
    where the last inequality is because $|\maS_{s,m}| = \binom{s}{m}^2 m!\le \pth{\frac{es}{m}}^m s^m = \pth{\frac{en}{1-\epsilon}}^m$.
    Therefore, when $m-1\ge \frac{6(1+\epsilon)\log\pth{\frac{en}{1+\epsilon}}}{\log\pth{1/(1-\rho)}}$, we have $\maQ\pth{\maT\ge \tau}\le \exp\pth{-\epsilon m\log\pth{\frac{en}{1-\epsilon}}}$.

We then upper bound $\maP\pth{\maT<\tau}$ under the alternative hypothesis $\maH_1$.
By~\eqref{eq:upbd-mat-tau}, we have that \begin{align*}
        \maP\pth{\maT<\tau}\le&~ \maP\pth{|S_{\pi^*}|<m}+\maP\pth{\en\pth{\maH_{\pi^*_m}^f}<\tau\big|\, |S_{\pi^*}|>m}\\
        \le&~ \exp\pth{-\frac{\epsilon^2 s^2}{2n}}+\exp\pth{-\frac{1}{2}\binom{m}{2}\pth{1-\log 2}},
    \end{align*}
    where the last inequality follows from~\eqref{eq:concentration_for_hyper_1} in Lemma~\ref{lem:Hypergeometric_distribution_prop} and $\frac{-\en\pth{\maH_{\pi^*_m}^f}}{1-\rho}\sim \chi^2\pth{\binom{m}{2}}$, and the concentration inequality for chi-square distribution~\eqref{eq:concentration_for_chisquare_uptl} in Lemma~\ref{lem:chisquare}. Since $m=\frac{(1-\epsilon)s^2}{n}$, there exists a universal constant $c_2>0$ such that, when $\frac{s^2}{n}\ge c_2$, we have $\maP\pth{\maT<\tau}\le 0.05$. Specifically, $\maP\pth{\maT<\tau} = o(1)$ when $s^2/n = \omega(1)$. Since $\maQ(\maT\ge \tau) = o(1)$ when $m-1\ge \frac{6(1+\epsilon)\log\pth{\frac{en}{1+\epsilon}}}{\log\pth{1/(1-\rho)}}$, there exists a universal constant $C_2$ such that, when $s^2 \ge C_2\pth{\frac{n\log n}{\log\pth{1/(1-\rho)}}\vee n}$, $\maP\pth{\maT<\tau}+\maQ\pth{\maT\ge \tau}\le 0.1$. Specifically, when $s^2/n = \omega(1)$, we have $\maP\pth{\maT<\tau} = o(1)$, and thus $\maP\pth{\maT<\tau}+\maQ\pth{\maT\ge \tau}=o(1)$.

\begin{remark}\label{rmk:subgaussian}
The Gaussian assumption on the weighted edges for $\beta_e(G_1)$ and $\beta_e(G_2)$ in Propositions~\ref{prop: upbd-overlap} and~\ref{prop: upbd-least-square} can be extended to the sub-Gaussian assumption.
    The main ingredients of our proof in these two Propositions are the analysis of the tail bound for the Gaussian distribution. We compute the Moment Generating Function (MGF) and use a standard Chernoff bound to bound $\maQ\pth{\maT\ge \tau}$. Indeed, if we relax the distribution assumption to a sub-Gaussian distribution, since the linear sum of sub-Gaussian random variables remains a sub-Gaussian random variable, the MGF in~\eqref{eq:mgf-strong} can be approximated by $\expect{\exp\pth{\lambda XY}}\le \frac{1}{\sqrt{1-c_1\lambda^2}}$ for some constant $c_1 \in \mathbb{R}$. Since the product of sub-Gaussian random variables is a sub-exponential random variable, the MGF in~\eqref{eq:mgf-strong} can be approximated by $\expect{\exp\pth{-\frac{\lambda}{2}\pth{X-Y}^2}}\le \frac{1}{\sqrt{1+c_2\lambda}}$ for some constant $c_2 \in \mathbb{R}$. For $\maP(\maT<\tau)$, the tail bound holds for sub-Gaussian as well. 
\end{remark}

\subsection{Proof of Proposition~\ref{prop:lwbd-sparse}}\label{apd:lwbd-aparse}
 
   For any $S\subseteq V(G_1)$ and $T\subseteq V(G_2)$ with $|S| = |T|$, define \begin{align}
       \label{eq:def_PST}\maP\pth{G_1,G_2,S,T} &=  \ti{\maP}\pth{G_1[S],G_2[T]}\prod_{e\notin \binom{S}{2}}\maQ_0\pth{\beta_e(G_1)}\prod_{e\notin \binom{T}{2}}\maQ_0\pth{\beta_e(G_2)},\\\label{eq:def_QST}
       \maQ\pth{G_1,G_2,S,T}&=\maQ\pth{G_1[S],G_2[T]}\prod_{e\notin \binom{S}{2}}\maQ_0\pth{\beta_e(G_1)}\prod_{e\notin \binom{T}{2}}\maQ_0\pth{\beta_e(G_2)},
   \end{align}
where $G[S]$ for any $S\subseteq V(G)$ denotes the induced subgraph with vertex set $S$ of $G$; $\ti{\maP}$ denotes the distribution of two random graphs follow fully correlated Gaussian Wigner model; $\maQ_0$ denotes the standard normal distribution. 
Recall that $S_{\pi^*}=V(G_1)\cap \pth{\pi^*}^{-1}\pth{V(G_2)}$ and $T_{\pi^*} = \pi^*\pth{V(G_1)}\cap V(G_2)$.
Indeed, $\maP\pth{G_1,G_2,S,T}$ denotes the distribution under $\maP$ when given $S_{\pi^*}=S$ and $T_{\pi^*}=T$. Besides, $\maQ\pth{G_1,G_2|S,T}$ and $\maQ\pth{G_1,G_2}$ are the same distribution for any $S\subseteq V(G_1),T\subseteq V(G_2)$ with $|S| = |T|$.

Since $\maP\pth{\cdot|\maE} = \frac{\maP\pth{\cdot,\maE}}{\maP\pth{\maE}} = \frac{\sum_{i=0}^{(1+\epsilon)s^2/n} \maP\pth{|S_{\pi^*}|=i}\maP\pth{\cdot \mid\,|S_{\pi^*}|=i}}{\maP\pth{\maE}}$ and $\TV\pth{\sum_{i}\lambda_i \maP_i,\maQ}\le \sum_{i}\lambda_i \TV\pth{\maP_i,\maQ}$ when $\sum_i \lambda_i=1$, we obtain
\begin{align}\label{eq:lwbd-sparse-1}
&~\TV\pth{\maP'(G_1,G_2),\maQ\pth{G_1,G_2}}\le \sum_{i=0}^{\frac{(1+\epsilon)s^2}{n}} \frac{\maP\pth{|S_{\pi^*}|=i}}{\maP\pth{\maE}}\cdot\TV\pth{\maP\pth{G_1,G_2\big|\,|S_{\pi^*}|=i},\maQ\pth{G_1,G_2}}.
\end{align}
For any $0\le i\le \frac{(1+\epsilon)s^2}{n}$ and $S\subseteq V(G_1),T\subseteq V(G_2)$ with $|S| = |T|=i$, by the data processing inequality (see, e.g., \cite[Section 3.5]{polyanskiy2025information}), we have
\begin{align}
\nonumber \TV\pth{\maP\pth{G_1,G_2\big|\,|S_{\pi^*}|=i},\maQ\pth{G_1,G_2}}\le&~ \TV\pth{\maP\pth{G_1,G_2,S,T},\maQ\pth{G_1,G_2,S,T}}\\\label{eq:lwbd-sparse-2}=&~\TV\pth{\ti{\maP}\pth{G_1[S],G_2[T]},\maQ\pth{G_1[S],G_2[T]}},
\end{align}
where the last equality follows from~\eqref{eq:def_PST}, ~\eqref{eq:def_QST} and the fact that $\TV\pth{X\otimes Z,Y\otimes Z} = \TV\pth{X,Y}$ for any distributions $X,Y,Z$ such that $Z$ is independent with $X$ and $Y$.


For the random graphs $G_1[S]$ and $G_2[T]$ with $S\subseteq V(G_1),T\subseteq V(G_2)$, and $|S| = |T|$, they follow the correlated Gaussian Wigner model with node set size $|S|$ under $\ti{\maP}$, while they are independent under $\maQ$. It follows from \cite[Theorem 1]{wu2023testing} that, when $\frac{|S|}{\log |S|}\le \frac{2}{\rho^2}$, the total variation distance $\TV\pth{\ti{\maP}\pth{G_1[S],G_2[T]},\maQ\pth{G_1[S],G_2[T]}}=o(1)$. We then verify the condition $\frac{|S|}{\log |S|}\le \frac{2}{\rho^2}$ for any $0\le |S|\le \frac{(1+\epsilon)s^2}{n}$.
In fact, since $s^2 \le \frac{n\log n}{2\log\pth{1/(1-\rho^2)}}$, we have \begin{align*}
    |S|\le \frac{(1+\epsilon) s^2}{n}\le \frac{(1+\epsilon)\log n}{2\log\pth{1/(1-\rho^2)}}\le \frac{2\log\pth{1/\rho^2}}{\rho^2},
\end{align*}
where the last inequality follows from $\log\pth{1/(1-\rho^2)}\ge \rho^2$, $\frac{\log n}{2}<\log\pth{1/\rho^2}$ and $\epsilon<1$. Therefore, we obtain $\frac{|S|}{\log |S|}\le \frac{\frac{2}{\rho^2}\log\pth{1/\rho^2}}{\log\pth{1/\rho^2}+\log\pth{2\log(1/\rho^2)} }\le \frac{2}{\rho^2}$, and thus \begin{align*}
    \TV\pth{\ti{\maP}\pth{G_1[S],G_2[T]},\maQ\pth{G_1[S],G_2[T]}} = o(1)
\end{align*}
for any $S\subseteq V(G_1),T\subseteq V(G_2)$ with $|S|=|T| \le  \frac{(1+\epsilon)s^2}{n}$. Combining this with~\eqref{eq:lwbd-sparse-1} and~\eqref{eq:lwbd-sparse-2}, we conclude that \begin{align*}
    \TV\pth{\maP'\pth{G_1,G_2},\maQ\pth{G_1,G_2}}&\le~\sum_{i=0}^{\frac{(1+\epsilon)s^2}{n}} \frac{\maP\pth{|S_{\pi^*}|=i}}{\maP\pth{\maE}}\cdot\TV\pth{\maP\pth{G_1,G_2\big|\,|S_{\pi^*}|=i},\maQ\pth{G_1,G_2}}\\
    &\le~\sum_{i=0}^{\frac{(1+\epsilon)s^2}{n}}\frac{\maP\pth{|S_{\pi^*}|=i}}{\maP\pth{\maE}}\cdot o(1) = o(1).
\end{align*}
Therefore, \begin{align}
    \nonumber \TV\pth{\maP(G_1,G_2),\maQ(G_1,G_2)} \overset{\mathrm{(a)}}{\le}&~ \TV\pth{\maP{(G_1,G_2)}, \maP'(G_1,G_2)}+\TV\pth{\maP'(G_1,G_2),\maQ(G_1,G_2)}\\\nonumber  \overset{\mathrm{(b)}}{\le}&~\TV\pth{\maP{(G_1,G_2,\pi)}, \maP'(G_1,G_2,\pi)}+\TV\pth{\maP'(G_1,G_2),\maQ(G_1,G_2)}\\\label{eq:upbd-tv-conditional}
    =&~ \maP\pth{(G_1,G_2,\pi)\notin \maE}+\TV\pth{\maP'(G_1,G_2),\maQ(G_1,G_2)}=o(1),
\end{align}
where $\mathrm{(a)}$ follows from the triangle inequality and $\mathrm{(b)}$ is derived by the data processing inequality (see, e.g., \cite[Section 3.5]{polyanskiy2025information}).

\subsection{Proof of Proposition~\ref{prop:lwbd-dense}}\label{apd:lwbd-dense}
Recall that the conditional distribution is defined as 
\begin{align*}
    \maP'(G_1,G_2,\pi) = \frac{\maP(G_1,G_2,\pi) \ones_{(G_1,G_2,\pi)\in \maE}}{\maP(\maE)} = (1+o(1)) \maP\pth{G_1,G_2,\pi} \ones_{(G_1,G_2,\pi)\in \maE},
\end{align*}
where the last inequality holds because $\maP\pth{\maE} = 1-o(1)$. 
By~\eqref{eq:TV-tsy-upbd} and~\eqref{eq:upbd-tv-conditional}, we have the following sufficient condition for the impossibility results:\begin{align}
    \label{eq:condi-second-intro}\mathbb{E}_\maQ \qth{\pth{\frac{\maP'(G_1,G_2)}{\maQ(G_1,G_2)}}^2} = 1+o(1)\Rightarrow \TV\pth{\maP',\maQ} = o(1)\Rightarrow\TV(\maP,\maQ) = o(1).
\end{align}

Recall the likelihood ratio in~\eqref{eq:def_of_likeli_ratio}. To compute the conditional second moment, we introduce an independent copy $\ti{\pi}$ of the latent permutation $\pi$ and express the square likelihood ratio as \begin{align*}
    \pth{\frac{\maP'(G_1,G_2)}{\maQ(G_1,G_2)}}^2 =&~(1+o(1))\mathbb{E}_\pi\qth{\frac{\maP(G_1,G_2|\pi)}{\maQ(G_1,G_2)} \ones_{(G_1,G_2,\pi)\in \maE}} \mathbb{E}_{\ti{\pi}} \qth{\frac{\maP(G_1,G_2|\ti{\pi})}{\maQ(G_1,G_2)}\ones_{(G_1,G_2,\ti{\pi})\in\maE}}
    \\=&~ (1+o(1))\mathbb{E}_{\pi \bot \ti{\pi}} \qth{\frac{\maP(G_1,G_2|\pi)}{\maQ(G_1,G_2)} \frac{\maP(G_1,G_2|\ti{\pi})}{\maQ(G_1,G_2)} \ones_{(G_1,G_2,\pi)\in \maE}\ones_{(G_1,G_2,\ti{\pi})\in \maE}}.
\end{align*}


Taking expectation for both sides under $\maQ$, the conditional second moment is given by  \begin{align}
    \mathbb{E}_\maQ \qth{\pth{\frac{\maP'(G_1,G_2)}{\maQ(G_1,G_2)}}^2} &= (1+o(1))
    \mathbb{E}_\maQ \qth{\mathbb{E}_{\pi \bot \ti{\pi}}
     \qth{\frac{\maP(G_1,G_2|\pi)}{\maQ(G_1,G_2)} \frac{\maP(G_1,G_2|\ti{\pi})}{\maQ(G_1,G_2)}\ones_{(G_1,G_2,\pi)\in \maE}\ones_{(G_1,G_2,\ti{\pi})\in \maE}}}\nonumber\\
    &=(1+o(1)) \mathbb{E}_{\pi \bot \ti{\pi}} \qth{\mathbb{E}_\maQ \qth{\frac{\maP(G_1,G_2|\pi)}{\maQ(G_1,G_2)} \frac{\maP(G_1,G_2|\ti{\pi})}{\maQ(G_1,G_2)}\ones_{(G_1,G_2,\pi)\in \maE}\ones_{(G_1,G_2,\ti{\pi})\in \maE}}}\nonumber\\ &=   (1+o(1)) \mathbb{E}_{\pi \bot \ti{\pi}} \qth{\ones_{(G_1,G_2,\pi)\in \maE}\ones_{(G_1,G_2,\ti{\pi})\in \maE}\mathbb{E}_\maQ \qth{\frac{\maP(G_1,G_2|\pi)}{\maQ(G_1,G_2)} \frac{\maP(G_1,G_2|\ti{\pi})}{\maQ(G_1,G_2)}}},\label{eq:dense_graph_condi_second_moment}
\end{align}
where the last equality holds since $\maE$ is independent with the edges in $G_1$ and $G_2$. 
Recall that $I^*= I^*(\pi,\ti{\pi})$ defined in~\eqref{eq:def_I^*}.
Since $I^*  = \cup_{C\in \sfC}\cup_{e\in C}\cup_{v\in V(e)\cap V(G_1)}v$ by the definition of $I^*$, we obtain that $\tbinom{I^*}{2} = \sum_{C\in \sfC} |C|$ by counting the edges induced by the vertices in $I^*$. Combining this with \eqref{eq:second_moment_kappa} and \eqref{eq:dense_graph_condi_second_moment}, we have that  \begin{align*}
        \mathbb{E}_\maQ \qth{\pth{\frac{\maP'(G_1,G_2)}{\maQ(G_1,G_2)}}^2} &=   (1+o(1)) \mathbb{E}_{\pi \bot \ti{\pi}} \qth{\ones_{(G_1,G_2,\pi)\in \maE}\ones_{(G_1,G_2,\ti{\pi})\in \maE}\mathbb{E}_\maQ \qth{\frac{\maP(G_1,G_2|\pi)}{\maQ(G_1,G_2)} \frac{\maP(G_1,G_2|\ti{\pi})}{\maQ(G_1,G_2)}}}\\
        &=  (1+o(1)) \mathbb{E}_{\pi \bot \ti{\pi}} \qth{\ones_{(G_1,G_2,\pi)\in \maE}\ones_{(G_1,G_2,\ti{\pi})\in \maE} \prod_{C\in \sfC}\pth{\frac{1}{1-\rho^{2|C|}}}}\\
        &\le (1+o(1)) \mathbb{E}_{\pi \bot \ti{\pi}} \qth{\ones_{(G_1,G_2,\pi)\in \maE}\ones_{(G_1,G_2,\ti{\pi})\in \maE} \prod_{C\in \sfC}\pth{\frac{1}{1-\rho^2}}^{|C|}}\\
        &=(1+o(1)) \mathbb{E}_{\pi \bot \ti{\pi}} \qth{\ones_{(G_1,G_2,\pi)\in \maE}\ones_{(G_1,G_2,\ti{\pi})\in \maE}  \pth{\frac{1}{1-\rho^2}}^{|I^*|(|I^*|-1)/2}},
    \end{align*}
where the inequality follows from $\frac{1}{1-\rho^{2x}}- \pth{\frac{1}{1-\rho^2}}^x=\frac{\pth{1-\rho^2}^x+\pth{\rho^2}^x-1}{\pth{1-\rho^{2x}}\pth{1-\rho^2}^x}\le \frac{1-\rho^2+\rho^2-1}{\pth{1-\rho^{2x}}\pth{1-\rho^2}^x}=0$ for any $0<\rho<1$ and $x\ge 1$.
Since $\prob{|I^*| = t}\le \pth{\frac{s}{n}}^{2t}$ by Lemma~\ref{lem:property_of_I} and $|I^*|\le |V(G_1)\cap \pi^{-1}(V(G_2))|\le \frac{(1+\epsilon)s^2}{n}$ if $(G_1,G_2,\pi),(G_1,G_2,\ti{\pi})\in \maE$, we obtain  
\begin{align}
    \nonumber \mathbb{E}_\maQ \pth{\frac{\maP'(G_1,G_2)}{\maQ(G_1,G_2)}}^2 &\le (1+o(1)) \mathbb{E}_{\pi \bot \ti{\pi}} \qth{\ones_{(G_1,G_2,\pi)\in \maE}\ones_{(G_1,G_2,\ti{\pi})\in \maE}  \pth{\frac{1}{1-\rho^2}}^{|I^*|(|I^*|-1)/2}}\\\nonumber
    &=(1+o(1)) \sum_{t=0}^{\frac{(1+\epsilon)s^2}{n}} \prob{|I^*| = t} \pth{\frac{1}{1-\rho^2}}^{t(t-1)/2}\\\label{eq:upbd-conditional-second}&\le(1+o(1)) \sum_{t=0}^{\frac{(1+\epsilon)s^2}{n}} \pth{\frac{s}{n}}^{2t} \pth{\frac{1}{1-\rho^2}}^{t(t-1)/2}.
\end{align}
Let $a_t\triangleq \pth{\frac{s}{n}}^{2t} \pth{\frac{1}{1-\rho^2}}^{t(t-1)/2}$. For any $t<\frac{(1+\epsilon)s^2}{n}$, we have 
\begin{align}
     \frac{a_{t+1}}{a_t}= \frac{s^2}{n^2} \pth{\frac{1}{1-\rho^2}}^t\le \frac{s^2}{n^2} \pth{\frac{1}{1-\rho^2}}^{\frac{(1+\epsilon)s^2}{n}}
    =\exp\pth{\log \pth{\frac{s^2}{n^2}}+\frac{(1+\epsilon)s^2}{n} \log\pth{\frac{1}{1-\rho^2}}}\label{eq:ratio_of_F}.
\end{align}
Since $s^2\le \frac{n\log n}{8\log\pth{1/(1-\rho^2)}}$, we obtain
\begin{align*}
    \frac{(1+\epsilon)s^2}{n} \log\pth{\frac{1}{1-\rho^2}}\le \frac{(1+\epsilon)\log n}{8}
\end{align*}
and 
\begin{align*}
    \log\pth{\frac{s^2}{n^2}} \le \log\pth{\frac{\log n}{8n\log\pth{1/(1-\rho^2)}}}\overset{\mathrm{(a)}}{\le} -\frac{1}{2}\log n+\log\pth{\frac{\log n}{8}},
\end{align*}
where $\mathrm{(a)}$ is because $\log\pth{\frac{1}{1-\rho^2}}\ge \log\pth{\frac{1}{1-n^{-1/2}}}\ge n^{-1/2}$.
Combining this with~\eqref{eq:ratio_of_F}, we obtain that $\frac{a_{t+1}}{a_t} \le \exp\pth{-\frac{(3-\epsilon)\log n}{8}+\log\pth{\frac{\log n}{8}}}\le n^{-1/4}$. Therefore, by~\eqref{eq:upbd-conditional-second}, \begin{align*}
    \mathbb{E}_\maQ \pth{\frac{\maP'(G_1,G_2)}{\maQ(G_1,G_2)}}^2 \le&~ (1+o(1)) \sum_{t=0}^{\frac{(1+\epsilon)s^2}{n}} a_t\\=&~(1+o(1))\sum_{t=0}^{\frac{(1+\epsilon)s^2}{n}} \pth{\frac{s}{n}}^{2t} \pth{\frac{1}{1-\rho^2}}^{t(t-1)/2} \le \frac{1+o(1)}{1-n^{-1/4}} = 1+o(1),
\end{align*}
which implies that $\TV(\maP,\maQ) = o(1)$ by \eqref{eq:condi-second-intro}.

\section{Proof of Lemmas}
\subsection{Proof of Lemma~\ref{lem:hypergro-def}}

Recall that $|V(G_1)| = |V(G_2)| = s$ and $V(G_1)\subseteq V(\mathbf{G}_1),V(G_2)\subseteq V(\mathbf{G}_2)$. For any $\pi \sim \maS_n$, we note that the event $|\pi(V(G_1))\cap V(G_2)|=t$ with $t\in [s]$ can be divided as:\begin{itemize}
    \item Picking $t$ vertices in $V(G_1),V(G_2)$ respectively and constructing the mapping between picked vertices. We have  $ {\tbinom{s}{t}^2 t!}$ options for this step.
    \item Mapping the remaining $s-t$ vertices in $V(G_1)$ to $V(\mathbf{G}_2)\backslash V(G_2)$. We have ${\tbinom{n-s}{s-t}(s-t)!}$ options for this step.
    \item Mapping $V(\mathbf{G}_1)\backslash V(G_1)$ to the remaining vertices in $V(\mathbf{G}_2)$. We have $(n-s)!$ options for this step.
\end{itemize}
Then, for any $t\le s$,  we have that \begin{align}\label{eq:hyper-distribution}
    \prob{|\pi(V(G_1))\cap V(G_2)| = t} = \frac{\tbinom{s}{t}^2t!\cdot \tbinom{n-s}{s-t}(s-t)!\cdot (n-s)!}{n!} = \frac{\tbinom{s}{t}\tbinom{n-s}{s-t}}{\tbinom{n}{s}},
\end{align}
which indicates that the size of  intersection set $|\pi(V(G_1))\cap V(G_2)|$ follows hypergeometric distribution $\HG(n,s,s)$ where $\pi\overset{\mathrm{Unif.}}{\sim} \maS_n$.

\subsection{Proof of Lemma~\ref{lem:partial_orbit_second_moment}}\label{apdsec:proof-of-partial-orbit}

For any $P = (e_1,\pi(e_1),e_2,\cdots,e_j,\pi(e_j))\in \sfP$ with $\ti{\pi}(e_2) = \pi(e_1),\cdots,\ti{\pi}(e_j) = \pi(e_{j-1})$, we have that \begin{align}
    \nonumber L_P &= \prod_{i=1}^j\ell(e_i,\pi(e_i)) \prod_{i=2}^{j}\ell(e_i,\ti{\pi}(e_i)) = \prod_{i=1}^j\ell(e_i,\pi(e_i)) \prod_{i=2}^{j}\ell(e_i,{\pi}(e_{i-1}))\\\label{eq:LP}&=\ell(e_1,\pi(e_1))\ell(\pi(e_1),e_2)\cdots\ell(\pi(e_{j-1}),e_j)\ell(e_j,\pi(e_j)). 
\end{align}
Under the distribution $\maQ$, it follows from~\eqref{eq:LP} that $L_P=\ell(B_0,B_1)\ell(B_1,B_2)\cdots\ell(B_{k-1},B_k)$ for some $k\in \mathbb{N}$ and $B_0, B_1,\cdots,B_k\overset{\mathrm{i.i.d.}}{\sim}\maN(0,1)$. Recall that \begin{align}\label{eq:likelihood-func}
    \ell(a,b) = \frac{\maP\pth{\beta_e(G_1)=a,\beta_{\pi(e)}(G_2)=b}}{\maQ\pth{\beta_e(G_1)=a,\beta_{\pi(e)}(G_2) = b}} = \frac{1}{\sqrt{1-\rho^2}} \exp\pth{\frac{-\rho^2(a^2+b^2)+2\rho ab}{2(1-\rho^2)}},\text{ for any }a,b\in\mathbb{R}.
\end{align}
Then,
     \begin{align*}
        &~\mathbb{E}_\maQ\qth{L_P} = \mathbb{E}_\maQ \qth{\ell(B_0,B_1)\ell(B_1,B_2)\cdots \ell(B_{k-1},B_k)}\\
        =&~\frac{1}{\pth{2\pi}^{(k+1)/2}\pth{(1-\rho^2)}^{k/2}}\int\cdots\int \exp\pth{\sum_{t=0}^{k-1}\frac{-\rho^2(b_t^2+b_{t+1}^2)+2\rho b_t b_{t+1}}{2(1-\rho^2)}}\exp\pth{\sum_{t=0}^{k}-\frac{b_t^2}{2}}\, db_0\cdots db_k \\
        =&~\frac{1}{\pth{2\pi}^{(k+1)/2}\pth{(1-\rho^2)}^{k/2}}\int\cdots\int \exp\pth{-\frac{\sum_{t=0}^{k-1} \pth{b_t-\rho b_{t+1}}^2}{2(1-\rho^2)} - \frac{b_{k}^2}{2}}\,db_0\cdots db_{k}=1,
    \end{align*}
where the last equality holds since
the transformation $B_t'\triangleq \frac{B_t-\rho B_{t+1}}{\sqrt{1-\rho^2}}$ for any $0\le t\le k-1$ yields that
$\mathbb{E}_\maQ\qth{L_P} = \frac{1}{(2\pi)^{(k+1)/2}} \int\cdots \int \exp\pth{-\frac{\sum_{t=0}^{k-1} b_k'^2}{2}-\frac{b_{k}^2}{2}}\,db_0'\cdots db_{k-1}' db_k = 1$.

For any $C = (e_1,\pi(e_1),e_2,\cdots,e_j,\pi(e_j))\in \sfC$ with $\ti{\pi}(e_2) = \pi(e_1),\cdots,\ti{\pi}(e_j) = \pi(e_{j-1})$ and $\ti{\pi}(e_1) = \pi(e_j)$, we denote $e_0 =e_j$ for notational simplicity. Then, we have that \begin{align*}
    L_C &= \prod_{i=1}^j\ell(e_i,\pi(e_i)) \prod_{i=1}^{j}\ell(e_i,\ti{\pi}(e_i)) = \prod_{i=1}^j\ell(e_i,\pi(e_i)) \prod_{i=1}^{j}\ell(e_i,{\pi}(e_{i-1}))\\&=\ell(e_1,\pi(e_1))\ell(\pi(e_1),e_2)\cdots\ell(\pi(e_{j-1}),e_j)\ell(e_j,\pi(e_j))\ell(\pi(e_j),e_1). 
\end{align*}
Then $L_C=\ell(B_1,B_2)\cdots\ell(B_{k-1},B_k)\ell(B_k,B_1)$ for  $k=2j$ and $B_1,\cdots,B_k\overset{\mathrm{i.i.d.}}{\sim} \maN(0,1)$.   
Denote $B_{k+1} = B_1$, we have that 
     \begin{align*}
        &\mathbb{E}_\maQ[L_C] = \mathbb{E}_\maQ\qth{\ell(B_1,B_2)\cdots \ell(B_{k-1},B_k)\ell(B_k,B_1)}\\=&\frac{1}{\pth{2\pi(1-\rho^2)}^{k/2}}\int\cdots \int \exp\pth{\frac{\sum_{t=0}^{k-1}-\rho^2\pth{b_t^2+b_{t+1}^2}+2\rho b_t b_{t+1}}{2(1-\rho^2)}}\exp\pth{\sum_{t=1}^k-\frac{b_t^2}{2}}\,db_1\cdots db_k\\=&\frac{1}{\pth{2\pi(1-\rho^2)}^{k/2}}\int\cdots \int \exp\pth{\frac{\sum_{t=0}^{k-1}-\pth{b_t-\rho b_{t+1}}^2}{2(1-\rho^2)}}\, db_1\cdots db_k.
    \end{align*}
Let $C_t\triangleq B_t-\rho B_{t+1}$ for any $1\le t \le k$. Then \begin{align*}
    \qth{C_1,C_2,\cdots,C_{k-1},C_k}^\top = \mathbf{J}_k \qth{B_1,B_2,\cdots,B_{k-1},B_k}^\top,
\end{align*}
where
\begin{align*}
    \mathbf{J}_k \triangleq \begin{bmatrix}
        1 & -\rho & 0&\cdots &0 \\0 &1&-\rho&\cdots  &0\\ 0&0&1&\cdots &0\\\vdots&\vdots& \vdots&\ddots &\vdots\\-\rho&0 &\cdots&0& 1
    \end{bmatrix}
\end{align*} and thus $\det\pth{\mathbf{J}_k}  =1-\rho^{k}$ (see, e.g., \cite[Section 3.2]{davis1979circulant}). Then, we obtain that \begin{align*}
    \mathbb{E}_\maQ[L_C] =\frac{1}{\pth{2\pi(1-\rho^2)}^{k/2} \det\pth{\mathbf{J}_k}} \int \cdots \int \exp\pth{\frac{\sum_{t=1}^k -c_t^2}{2(1-\rho^2)}}\, dc_1\cdots dc_k = \frac{1}{1-\rho^k} = \frac{1}{1-\rho^{2|C|}}.
\end{align*}

\subsection{Proof of Lemma~\ref{lem:property_of_I}}
    Let $I'\triangleq \argmax_{I\subseteq V(G_1),\pi(I) = \ti{\pi}(I)}|I|$, we first show that $I' = I^*$. On the one hand, since $\pi(I') = \ti{\pi}(I')$, we have $\pi\pth{\tbinom{I'}{2}} = \ti{\pi}\pth{\tbinom{I'}{2}}$.
    Recall that the connected components of the correlated functional digraph in Definition~\ref{def:correlated-functional} consist of paths and cycles. For any path $P\in \sfP$,
    we note that $\pi\pth{P\cap V(G_1)}\neq \ti{\pi}\pth{P\cap V(G_1)}$, and thus $\pi\pth{\binom{P\cap V(G_1)}{2}} \neq \ti{\pi}\pth{\binom{P\cap V(G_1)}{2}}$.
    For any cycle $C\in \sfC$,
    we note that $\pi\pth{C\cap V(G_1)}= \ti{\pi}\pth{C\cap V(G_1)}$, and thus $\pi\pth{\binom{C\cap V(G_1)}{2}} = \ti{\pi}\pth{\binom{C\cap V(G_1)}{2}}$.
    Therefore, $\tbinom{I'}{2}\subseteq \cup_{C\in \sfC}\cup_{e\in C\cap E(G_1)}e$. By the definition of $I^*$, we obtain $I'\subseteq I^*$. On the other hand, for any $C\in \sfC$, since $\pi\pth{\cup_{e\in C\cap E(G_1)} e} = \ti{\pi}\pth{\cup_{e\in C\cap E(G_1)} e}$ by the definition of a cycle and $C\cap C'=\emptyset$ for any $C\neq C'\in \sfC$, we have that \begin{align*}
        \pi\pth{\cup_{C\in\sfC}\cup_{e\in C\cap E(G_1)} e} = \ti{\pi}\pth{\cup_{C\in \sfC} \cup_{e\in C\cap E(G_1)}e}.
    \end{align*}
    Therefore, we have $\pi\pth{\cup_{C\in\sfC}\cup_{e\in C} \cup_{v\in v(e)\cap V(G_1)}v} = \ti{\pi}\pth{\cup_{C\in \sfC} \cup_{e\in C}\cup_{v\in v(e)\cap V(G_1)}v}$, which implies $\pi(I^*) = \ti{\pi}(I^*)$. Since $I^*\subseteq V(G_1)$, by the definition of $I'$,  we conclude that $I^*\subseteq I'$.  Therefore, we have $I^* = I' = \argmax_{I\subseteq V(G_1),\pi(I) = \ti{\pi}(I)} |I|$.

    For any $t\le s$, by the union bound, we obtain  
    \begin{align}
        \nonumber \prob{|I^*|=t}&\le~ \prob{\exists A\subseteq V(G_1),|A| = t,\pi(A) = \ti{\pi}(A)\subseteq V(G_2)}
\\\label{eq:upbd-I^*-size}&\le~ \tbinom{s}{t}\prob{A\subseteq V(G_1),|A| = t,\pi(A) = \ti{\pi}(A)\subseteq V(G_2)}.
    \end{align}
For any fixed set $A \subseteq V(G_1)$ with $|A| = t$ and $\pi(A) = \tilde{\pi}(A) \subseteq V(G_2)$, we first choose a set $B \subseteq V(G_2)$ with $|B| = t$, and set $\pi(A) = \tilde{\pi}(A) = B$. There are $\binom{s}{t}$ ways to choose $B$, and $t!^2$ ways to map $\pi(A) = \tilde{\pi}(A) = B$. For the remaining vertices in $V(G_1)$, there are $(n-t)!^2$ ways to map them under $\pi$ and $\tilde{\pi}$. Therefore,
\begin{align}
 \tbinom{s}{t}\prob{A\subseteq V(G_1),|A| = t,\pi(A) = \ti{\pi}(A)\subseteq V(G_2)}
=\tbinom{s}{t}\cdot \frac{1}{(n!)^2}\tbinom{s}{t}t!^2(n-t)!^2 \label{eq:upper_bound_for_size_of_I*}
\le \pth{\frac{s}{n}}^{2t},
\end{align}
where the last inequality is due to the fact that $\tbinom{s}{t}\cdot \frac{1}{(n!)^2}\tbinom{s}{t}t!^2(n-t)!^2  = \qth{\frac{s(s-1)\cdots(s-t+1)}{n(n-1)\cdots(n-t+1)}}^2$
and  for any $i = 1,\cdots,t-1$, $\frac{s-i}{n-i}\le \frac{s}{n}$. 
Combining this with~\eqref{eq:upbd-I^*-size}, we obtain $\prob{|I^*|=t}\le \pth{\frac{s}{n}}^{2t}$.

\section{Auxiliary Results}

\subsection{Concentration Inequalities for Gaussian}

\begin{lemma}[Hanson-Wright inequality]\label{lem:Hanson-Wright}
    Let $X,Y\in \mathbb{R}^n$ be standard Gaussian vectors such that the pairs $(X_i,Y_i)\sim \gaussianrho$ are independent for $i=1,\cdots,n$. Let $M_0\in \mathbb{R}^{n\times n}$ be any deterministic matrix. There exists some universal constant $c_0>0$ such that \begin{align*}
        \prob{\left| X^\top M_0 Y -\rho \mathrm{Tr}(M_0)\right|\ge c_0\pth{\Vert M_0\Vert_F \sqrt{\log(1/\delta)}\vee\Vert M_0\Vert_2\log(1/\delta)}}\le \delta.
    \end{align*} 
\end{lemma}

\begin{proof}
    Note that $X^{\top} M_0Y = \frac{1}{4}(X+Y)^\top M_0(X+Y)-\frac{1}{4}(X-Y)^\top M_0(X-Y)$ and \begin{align*}
        \expect{(X+Y)^\top M_0(X+Y)} = (2+2\rho)\mathrm{Tr}(M_0),\expect{(X-Y)^\top M_0(X-Y)} = (2-2\rho)\mathrm{Tr}(M_0).
    \end{align*}
    By Hanson-Wright inequality \cite{hanson1971bound}, there exists some universal constant $c_0$ such that \begin{align*}
        &\prob{\left| \frac{1}{4}(X+Y)^\top M_0(X+Y) - \frac{2+2\rho}{4}\mathrm{Tr}(M_0)\right| \ge \frac{c_0}{2}\pth{\Vert M_0\Vert_F \sqrt{\log(1/\delta)}\vee\Vert M_0\Vert_2\log(1/\delta)}}\le \frac{\delta}{2},\\&\prob{\left| \frac{1}{4}(X-Y)^\top M_0(X-Y) - \frac{2-2\rho}{4}\mathrm{Tr}(M_0)\right| \ge \frac{c_0}{2}\pth{\Vert M_0\Vert_F \sqrt{\log(1/\delta)}\vee\Vert M_0\Vert_2\log(1/\delta)}}\le \frac{\delta}{2}
    \end{align*} for any $\delta>0$.
    Consequently, \begin{align*}
        &~\prob{\left| X^\top M_0 Y -\rho \mathrm{Tr}(M_0)\right|\ge c_0\pth{\Vert M_0\Vert_F \sqrt{\log(1/\delta)}\vee\Vert M_0\Vert_2\log(1/\delta)}}\\ \le &~\prob{\left| \frac{1}{4}(X+Y)^\top M_0(X+Y) - \frac{2+2\rho}{4}\mathrm{Tr}(M_0)\right| \ge \frac{c_0}{2}\pth{\Vert M_0\Vert_F \sqrt{\log(1/\delta)}\vee\Vert M_0\Vert_2\log(1/\delta)}}\\+&~ \prob{\left| \frac{1}{4}(X-Y)^\top M_0(X-Y) - \frac{2-2\rho}{4}\mathrm{Tr}(M_0)\right| \ge \frac{c_0}{2}\pth{\Vert M_0\Vert_F \sqrt{\log(1/\delta)}\vee\Vert M_0\Vert_2\log(1/\delta)}}\le \delta.
    \end{align*}
\end{proof}

\subsection{Concentration Inequalities for Chi-Squared Distribution}
\begin{lemma}[Chernoff's inequality for Chi-squared distribution]\label{lem:chisquare}
    Suppose $\xi\sim \chi^2(n)$. Then, for any  $\delta>0$, we have \begin{align}
        \prob{\xi>(1+\delta) n}&\le  \exp\pth{-\frac{n}{2}\pth{\delta-\log\pth{1+\delta}}}\label{eq:concentration_for_chisquare_uptl},\\\prob{\xi<(1-\delta) n}&\le \exp\pth{-\frac{n}{2}\pth{-\delta-\log(1-\delta)}}\label{eq:concentration_for_chisquare_lwtl}.
    \end{align}
\end{lemma}
\begin{proof}
    The results follow from Theorems 1 and 2 in \cite{ghosh2021exponential}.
\end{proof}

\subsection{Concentration Inequalities for Hypergeometric Distribution}

\begin{lemma}[Concentration inequalities for Hypergeometric distribution]\label{lem:Hypergeometric_distribution_prop}
For $\eta \sim \HG(n,s,s)$ and any $\epsilon>0$, we have 

\begin{align}
    \label{eq:concentration_for_hyper_2}\prob{\eta \ge \frac{(1+\epsilon)s^2}{n}}&\le \exp\pth{-\frac{\epsilon^2 s^2}{(2+\epsilon) n}}\wedge \exp\pth{-\frac{\epsilon^2 s^3}{n^2}},\\\label{eq:concentration_for_hyper_1}\prob{\eta \le\frac{(1-\epsilon)s^2}{n}} &\le \exp\pth{-\frac{\epsilon^2 s^2}{2n}}\wedge \exp\pth{-\frac{\epsilon^2 s^3}{n^2}}.
\end{align}

\end{lemma}

\begin{proof}
Denote $\xi \sim \Bin\pth{s,\frac{s}{n}}$, by Theorem 4 in \cite{hoeffding1994probability}, for any continuous and convex function $f$, we have 
\begin{align*}
    \expect{f(\eta)} \le \expect{f(\xi)}.
\end{align*}
We note the function $f(x) =\exp\pth{\lambda x}$ is continuous and convex for any $\lambda\in \mathbb{R}$. Therefore, we have $\expect{\exp\pth{\lambda \eta}}\le \expect{\exp\pth{\lambda \xi}}$ for any $\lambda \in \mathbb{R}$, and thus the Chernoff bound for $\xi$ remains valid for $\eta$.
Combining this with Theorems 4.4 and 4.5 in  \cite{mitzenmacher_probability_2005}, we have 
\begin{equation*}
    \prob{\eta\ge \frac{(1+\epsilon) s^2}{n}}\le \exp\pth{-\frac{\epsilon^2 s^2}{(2+\epsilon)n}},\quad\prob{\eta \le \frac{(1-\epsilon)s^2}{n}}\le \exp\pth{-\frac{\epsilon^2 s^2}{2n}}.
\end{equation*}
By Hoeffding's inequallity \cite{hoeffding1994probability}, we also have 
\begin{equation*}
    \prob{\eta\ge \frac{(1+\epsilon) s^2}{n}}\le \exp\pth{-\frac{\epsilon^2 s^3}{n^2}},\quad \prob{\eta \le \frac{(1-\epsilon)s^2}{n}}\le \exp\pth{-\frac{\epsilon^2 s^3}{n^2}}.
\end{equation*}
Therefore, we finish the proof of Lemma \ref{lem:Hypergeometric_distribution_prop}.
\end{proof}

\section{Additional Experiments}\label{apdsec:additional-exp}
We provide a simple illustration on how our algorithm can be applied on real dataset. We conduct an experiment on Freeman’s EIES networks~\cite{freeman1979networkers}, a small dataset of 46 researchers, where edge weights represent communication strength at two time points. 
We apply our method to test for correlation between these two temporal networks. We examine how sample size affects privacy protection by analyzing the normalized similarity score, defined as the similarity score $\en(\maH_\pi^f)$ 
divided by $\binom{s}{2}$. Indeed, a lower score suggests weaker correlation and greater support for the null hypothesis of independence.

We apply our algorithm to the EIES dataset at different sample sizes, 
$s = 10, 20, 40$ and compute the corresponding normalized similarity scores: -1.066, -0.905, and -0.651. The scores increase with sample size, indicating stronger detected correlation. The lower scores at small sample sizes reflect failed correlation detection, quantifying the reduction in re-identification risk.

\end{document}